\documentclass[11pt]{amsart}
\usepackage{enumitem}
\usepackage{amsthm}
\usepackage[paperheight=11in,paperwidth=8.5in,left=1in,top=1.25in,right=1in,bottom=1.25in,]{geometry}
 \usepackage{relsize}
\usepackage{rotating}
\usepackage[table]{xcolor}
\usepackage{multirow}

\usepackage[super]{nth}
\usepackage[none]{hyphenat}
\usepackage{hyperref}
\hypersetup{
	colorlinks,
	citecolor=black,
	filecolor=black,
	linkcolor=black,
	urlcolor=black
}
\usepackage[capitalize, nameinlink, noabbrev, nosort]{cleveref}
\usepackage{graphicx, amssymb}
\usepackage{epstopdf} 
\usepackage{youngtab}
\usepackage[boxsize=.35 em]{ytableau}
\usepackage{comment}
\usepackage{setspace}
\usepackage{multicol}
\usepackage{manfnt}
\usepackage{tikz}   
\usetikzlibrary{arrows,calc,intersections,matrix,positioning,through}
\usepackage{tikz-cd}
\usepackage{comment}
\usepackage{diagbox}
\usepackage{blkarray}
\tikzset{commutative diagrams/.cd,every label/.append style = {font = \normalsize}}

\usepgflibrary{shapes.geometric}
\usetikzlibrary{shapes.geometric}
\usetikzlibrary{backgrounds}

\DeclareMathOperator{\Ext}{Ext}

\DeclareMathOperator{\Gr}{Gr}
\DeclareMathOperator{\Mat}{Mat}

\DeclareMathOperator{\sgn}{sgn}

\DeclareMathOperator{\cdes}{cDes_L}

\DeclareMathOperator{\var}{var}

\DeclareMathOperator{\PT}{PT}

\DeclareMathOperator{\area}{area}
\DeclareMathOperator{\garea}{gr-area}
\DeclareMathOperator{\Res}{Res}
\DeclareMathOperator{\sh}{sh}
\DeclareMathOperator{\Sh}{Sh}
\newcommand{\rf}[1]{\hyperref[#1]{(\ref*{#1})}}

\newcommand{\psimp}{\tilde{\Delta}}

\newcommand{\hatG}{\hat{G}}

\newcommand{\shuff}{{\sh^{(ij)}_{n}}}

\newcommand{\Grk}{\Gr_{k,n}^{\ge 0}}

\def\AA{\mathcal{A}_{n,k,m}(Z)}
\newcommand{\Ank}{\mathcal{A}_{n, k, 2}({Z})}

\def\tZ{\tilde{Z}}

\def\4biddenprop{4-coindependent}

\newcommand{\llrr}[1]{\langle\!\langle #1 \rangle\!\rangle}

\newcommand{\gt}[1]{Z_{#1}}
\newcommand{\gto}[1]{Z_{#1}^\circ}
\def\M{\mathcal{M}}

\DeclareMathOperator{\Vol}{Vol}
\DeclareMathOperator{\pt}{pt}
\def\unistr{\widehat{\Gr}_{2,n}^\circ}

\newcommand{\C}{\mathbb{C}}

\newcommand{\A}{\mathcal{A}}

\newcommand{\T}{\mathcal{T}}

\def\R{\mathbb{R}}

\newcommand{\htile}[1]{\Gamma_{#1}}
\newcommand{\ptile}[1]{\tilde{\Gamma}_{#1}}
\newcommand{\atile}[1]{Z_{#1}}
\newcommand{\simp}[1]{\Delta_{(#1)}}
\newcommand{\cham}[1]{\Delta_{(#1)}^Z}
\newcommand{\gcham}[1]{\Delta_{(#1)}^\mathcal{G}}
\def\sdiv{{\sigma}}
\def\tdiv{{\tau}}

\newtheorem{mainthm}{Theorem}

\newtheorem*{theorem*}{Theorem}
\newtheorem*{proposition*}{Proposition}
\newtheorem*{corollary*}{Corollary}
\newtheorem{theorem}{Theorem}[section]
\newtheorem{lemma}[theorem]{Lemma}
\newtheorem{proposition}[theorem]{Proposition}
\newtheorem{corollary}[theorem]{Corollary}

\newtheorem{question}[theorem]{Question}
\theoremstyle{definition}
\newtheorem{definition}[theorem]{Definition}
\newtheorem{example}[theorem]{Example}

\newtheorem{remark}[theorem]{Remark}
\newtheorem{notation}[theorem]{Notation}
\newtheorem{claim}[theorem]{Claim}

\setlist[itemize]{leftmargin=*}
\setlist[enumerate]{leftmargin=*}

\begin{document}
\begin{abstract}
The \emph{amplituhedron} $\mathcal{A}_{n,k,m}$ is a geometric object introduced in the 
context of scattering amplitudes in $\mathcal{N}=4$ super Yang Mills.  It generalizes 
the positive Grassmannian (when $n=k+m$), cyclic polytopes (when $k=1$), 
and the bounded complex of the cyclic hyperplane arrangement (when $m=1$). Of substantial interest are the \emph{tilings} of the amplituhedron, which are analogous to triangulations of a polytope.
In \cite{karp2020decompositions}, it was observed that the known tilings of 
$\mathcal{A}_{n,k,2}$ have cardinality ${n-2 \choose k}$ and the known tilings of 
$\mathcal{A}_{n,k,4}$ have cardinality the \emph{Narayana number} $\frac{1}{n-3}{n-3 \choose k+1}{n-3 \choose k}$;
generalizing these observations, \cite{karp2020decompositions} conjectured that for even $m$ the tilings of $\mathcal{A}_{n, k,m}$ have cardinality
the \emph{MacMahon number}, the number of plane partitions which fit inside a $k \times (n-k-m) \times \frac{m}{2}$ box.
We refer to this prediction as the \emph{Magic Number Conjecture}. 
In this paper we prove the Magic Number Conjecture for the $m=2$ amplituhedron:
that is, we show that 
each tiling of $\mathcal{A}_{n,k,2}$ has cardinality ${n-2 \choose k}.$ 
We prove this by showing that all  positroid tilings of the hypersimplex
$\Delta_{k+1,n}$ have cardinality ${n-2 \choose k},$ then applying
 \emph{T-duality}.  In addition, we give
combinatorial necessary conditions for tiles to form a tiling of $\mathcal{A}_{n,k,2}$;
 we give volume formulas for 
 \emph{Parke-Taylor polytopes} and certain positroid polytopes 
 in terms of circular extensions of \emph{cyclic partial orders}; and we prove
	new variants of the 
	classical \emph{Parke-Taylor identities}.
 
\end{abstract}
	
\title
[The magic number conjecture for $\mathcal{A}_{n,k,2}$ and Parke-Taylor identities]
{The magic number conjecture for the $m=2$ amplituhedron and Parke-Taylor identities}

	\author[M. Parisi]{Matteo Parisi}
	\address{CMSA, Harvard University, Cambridge, MA; Institute for Advanced Study, Princeton, NJ;}
 \email{mparisi@cmsa.fas.harvard.edu}
	
 \author[M. Sherman-Bennett]{Melissa Sherman-Bennett}
	\address{Department of Mathematics, MIT, Cambridge, MA}
	\email{msherben@mit.edu}
	\author[R. Tessler]{Ran Tessler}
	\address{Department of Mathematics, Weizmann Institute of Science, Israel}
	\email{ran.tessler@weizmann.ac.il}
	\author[L. Williams]{Lauren Williams}
	\address{Department of Mathematics, Harvard University, Cambridge, MA}
	\email{williams@math.harvard.edu}
	\maketitle

	\setcounter{tocdepth}{1}
	\tableofcontents

\section{Introduction}

The (tree) \emph{amplituhedron} $\mathcal{A}_{n,k,m}(Z)$ is the image of the
positive Grassmannian $\Gr_{k,n}^{\scriptscriptstyle\geq 0}$ under the
\emph{amplituhedron map} $\tilde{Z}: \Gr_{k,n}^{\scriptscriptstyle\geq 0} \to \Gr_{k,k+m}$, a map induced by matrix multiplication by 
a positive matrix $Z\in \Mat_{n,k+m}^{>0}$.  The amplituhedron was introduced by
Arkani-Hamed and Trnka \cite{arkani-hamed_trnka} in order  to give a
geometric  interpretation of
\emph{scattering amplitudes} in $\mathcal{N}=4$ super Yang--Mills theory (SYM);
more specifically, they reformulated the BCFW recurrence for computing
scattering amplitudes as giving a \emph{tiling} of the $m=4$ amplituhedron.
Here, a \emph{tiling} of $\AA$ 
comes from a collection of $km$-dimensional cells
of $\Gr_{k,n}^{\geq 0}$ on which $\tilde{Z}$ is injective, such that the images of the cells are disjoint and cover a dense subset of the amplituhedron.  The notion of tiling can be thought of as a generalization of the notion of triangulation of a polytope.\footnote{We don't require tiles to intersect in a common ``face". For polytopes, this is weaker than a triangulation and it is sometime referred to as a \emph{dissection}.}

While the case $m=4$ is most directly relevant to physics, 
the amplituhedron $\mathcal{A}_{n,k,m}(Z)$
makes sense for any positive $n,k,m$ such that $k+m \leq n$, and
has a very rich geometric and combinatorial
structure.  It  generalizes cyclic polytopes (when $k=1$),
cyclic hyperplane arrangements \cite{karpwilliams}
(when $m=1$), and the positive Grassmannian (when $k=n-m$), and it
is connected to
the hypersimplex and the positive tropical Grassmanian \cite{LPW,PSW} (when $m=2$). This paper will focus on the case $m=2$.  In this case, the amplituhedron $\mathcal{A}_{n,k,2}$
is also closely
related to some scattering amplitudes, correlators of determinant operators and form factors in planar $\mathcal{N}=4$ super Yang--Mills theory \cite{Kojima:2020tjf,Caron-Huot:2023wdh,Basso:2023bwv}.

In \cite{karp2020decompositions} it was observed that the known tilings of
$\mathcal{A}_{n,k,2}(Z)$ have cardinality ${n-2 \choose k}$,  the known tilings of
$\mathcal{A}_{n,k,4}(Z)$ have cardinality the Narayana number $\frac{1}{n-3}{n-3 \choose k+1}{n-3 \choose k}$,
and all tilings of 
$\mathcal{A}_{n,1,m}$ for even $m$
have cardinality ${n-1- \frac{m}{2} \choose \frac{m}{2}}$.\footnote{This last statement
comes from the fact that 
every triangulation of the cyclic polytope
$C(n,m)$ contains exactly ${n-1- \frac{m}{2} \choose \frac{m}{2}}$ simplices when $m$ is even
\cite{bayer_93, rambau_97} and that for $C(n,m)$ tilings and triangulations coincide \cite{OT10}.}
\cite[Conjecture 8.1]{karp2020decompositions} generalized these observations 
by predicting that 
when $m$ is even, 
the amplituhedron $\mathcal{A}_{n,k,m}$ has a tiling with cardinality
$$M_{n,k,m}:=M\left(k, n-k-m, \frac{m}{2}\right), \ \ \text{ where } \ \ 
M(a,b,c):=\prod_{i=1}^a \prod_{j=1}^b \prod_{\ell=1}^c \frac{i+j+\ell-1}{i+j+\ell-2}$$ is
the \emph{MacMahon number}.\footnote{\cite[Remark 8.2]{karp2020decompositions} pointed
out that the conjecture could be generalized to all $m$ using 
$M(k,n-k-m, \lfloor \frac{m+1}{2} \rfloor).$}
 This number has many remarkable interpretations: $M(a,b,c)$ counts the number of \emph{plane partitions} which fit inside an $a \times b \times c$ box, collections of $c$ noncrossing lattice paths 
inside an $a\times b$ rectangle, rhombic tilings of a hexagon with side lengths $(a,b,c,a,b,c)$, perfect matchings of a honeycomb lattice with parameters 
$a,b,c$, \emph{Kekul\'e structures} of a hexagon-shaped benzenoid with parameters 
$a,b,c$, and the dimension of the degree $c$ component of the homogeneous coordinate
ring $\C[\Gr_{a,a+b}]$, see \cite{macmahon, gordon_davison_52, cyvin_1986, kekule_58, kekule, bodroza_gutman_cyvin_tosic_1988, hodge43}. 
(Plane partitions also naturally appear in the computation of the Euler characteristic of the Hilbert scheme of points in a $3$-fold \cite{MR1382733}, and in the study of Donaldson-Thomas invariants \cite{maulik2006gromov,maulik2006gromov2}.)
 After \cite{karp2020decompositions} appeared, \cite{Galashin_Lam_2020}  restated the conjecture but specified that \emph{each} tiling
should have this cardinality. 
We refer to the prediction that tilings of the amplituhedron $\mathcal{A}_{n,k,m}(Z)$ should have cardinality $M_{n,k,m}$ as the \emph{Magic Number Conjecture}. 

In this paper we prove the Magic Number Conjecture for the $m=2$ amplituhedron;
the following result appears as
\cref{cor:magic-number}.

\begin{mainthm}[Magic Number Theorem for $\Ank$]
Suppose a collection $\mathcal{C}$ of cells of $\Gr_{k,n}^{\geq 0}$
gives rise to a tiling of $\Ank$ for all $Z\in \Mat_{n,k+m}^{>0}$.
Then $\mathcal{C}$ has cardinality $M_{n,k,2}=\binom{n-2}{k}$.
\end{mainthm}

Recall that the \emph{hypersimplex} $\Delta_{k+1,n}$ is the convex
hull of all $0/1$ vectors $e_I\in \R^n$, where $I$ is a $(k+1)$-element subset of $[n]$.  It is the \emph{uniform matroid polytope} 
and can also be viewed as the image of the positive Grassmannian 
$\Gr_{k+1,n}^{\geq 0}$ under the \emph{moment map}.
To prove the Magic Number Conjecture for $m=2$, we prove 
that each tiling of the hypersimplex $\Delta_{k+1,n}$ by 
positroid polytopes 
has cardinality $\binom{n-2}{k}$, then apply the T-duality theorem from 
\cite[Theorem 11.6]{PSW} (which appears here as \cref{thm:Tduality}).
\begin{mainthm}\label{mainthm:B}
Every positroid tiling of the hypersimplex $\Delta_{k+1, n}$ consists of $\binom{n-2}{k}$ tiles. 
\end{mainthm}
The above theorem generalizes the result from \cite{SW21} that every finest regular positroid subdivision of $\Delta_{k+1,n}$ has cardinality 
$\binom{n-2}{k}.$ 
Generalizing \cref{mainthm:B}, we also prove that 
if $\Gamma_{\mathcal{M}}$ is a full-dimensional positroid polytope,
all of its tilings have the same cardinality, see \cref{cor:positroidpolytopes}.

We also 
establish some necessary conditions for a collection of cells of $\Gr_{k,n}^{\geq 0}$ to give rise to a tiling of $\Ank$ (equivalently, $\Delta_{k+1, n}$). These results appear in \cref{lem:tilesside} and \cref{prop:covering-ngon-necessary}, and use the combinatorics of \emph{bicolored subdivisions} (cf. \cref{def:bicolored}).

 Our third main result is a formula for the volume of certain polytopes as the number of \emph{circular extensions} of a \emph{partial cyclic order}. Circular extensions and partial cyclic orders are analogues of linear extensions and partial orders.
 The following result appears as \cref{cor:vol-of-tile}; see \cref{prop:volume} for a generalization.
 \begin{mainthm}
Let $\sigma$ be a \emph{bicolored subdivision} (see \cref{def:bicolored}), let $\Gamma_{\sigma}$ be the corresponding positroid polytope (see \cref{thm:hypersimplex-tiles-ineq-facets}), and let $C_{\sigma}$ be the corresponding partial cyclic order (see \cref{def:cyclic-from-perm-subdiv}).  Then the normalized volume of $\Gamma_{\sigma}$ equals the number of circular extensions of $C_{\sigma}.$
 \end{mainthm}

We note that \cite{AJR_cyclicOrders} also studied some polytopes contained in $[0,1]^n$ whose volumes can be 
computed in terms of circular extensions of partial cyclic orders; they were motivated in part by 
\cite[Exercise 4.56(d)]{EC1}.  The results of \cite{AJR_cyclicOrders} were subsequently generalized in \cite{Morales}.  In general our polytopes are different from the polytopes of \cite{AJR_cyclicOrders} although
in a very special case they agree, see 
\cref{rem:AJVR}.

A key tool in our proof of the Magic Number Theorem is the \emph{Parke-Taylor} function\footnote{From the seminal paper of the physicists Stephen J. Parke and T.R. Taylor \cite{PTs}.} associated to a permutation.

\begin{definition}\label{def:PT}
Let $w=w_1 \dots w_n \in S_n$ be a permutation on $[n]$. 
The \emph{Parke-Taylor function} of $w$ is
\begin{equation}
    \PT(w):=\frac{1}{P_{w_1 \, w_2} P_{w_2 \, w_3} \ldots P_{w_{n} \, w_1}},
\end{equation}
where the $P_{ij}$ are Pl\"ucker coordinates on $\Gr_{2,n}$, with the convention that Pl\"ucker coordinates are antisymmetric in their indices.
\end{definition}

For example, if $w=2143$ then 
$$\PT(w) = \frac{1}{P_{21} P_{14} P_{43} P_{32}} = 
-\frac{1}{P_{12} P_{14} P_{34} P_{23}}.$$

We view the Parke-Taylor function $\PT(w)$ 
as a rational function on (the affine cone over) the Grassmannian $\Gr_{2,n}$.  It is well-defined
on (the cone over) the top-dimensional matroid stratum of $\Gr_{2,n}$, that is, the locus 
where no Pl\"ucker coordinates vanish. 
Parke-Taylor functions have numerous connections e.g. with the cohomology of the moduli space $\mathcal{M}_{0,n}$ of $n$ points on the Riemann sphere in relation with \emph{scattering equations} \cite{Cachazo:2013hca} and \emph{Lie polynomials} \cite{Frost:2019fjn}.

As a corollary of our work, we obtain many relations for Parke-Taylor functions (cf. \cref{thm:gen-parke-taylor}), one for each \emph{tricolored subdivision} of an $n$-gon (cf. \cref{def:tricolored}), in terms of circular extensions of certain cyclic partial orders. Moreover, we prove the known \emph{shuffle identities} for Parke-Taylor functions as below.

\begin{proposition*}
Let $I\subseteq [n-1]$ and fix a permutation $u=u_1 \dots u_r$ of $I$ and a permutation $v=v_1 \dots v_{n-r-1}$ of $[n-1]\setminus I$. Let $\Sh_n(u,v)$ be the set of permutations in $S_n$ in which $u_1,\ldots,u_r$ appear in order and $v_1,\ldots,v_{n-r-1}$ appear in order, and which send $n$ to $n$. Then we have
\begin{equation*}
 \sum_{w\in \Sh_n(u,v)} \PT(w) = 0.   
\end{equation*}
\end{proposition*}

The result above appears later as \cref{prop:shufflePT}.    
In the special case that $I=\{1\}$,
this result recovers the ``classical'' \emph{$U(1)$-decoupling identity} for Parke-Taylor functions (cf. \cref{cor:U1}).

\bigskip

The structure of this paper is as follows. In \cref{sec:background} we provide
background on the positive Grassmannian, the hypersimplex, and the amplituhedron, including a well-known triangulation of the hypersimplex into $w$-simplices. In \cref{sec:combtiles} we give a simple combinatorial characterization of exactly which $w$-simplices lie in a given tile for the hypersimplex. In \cref{sec:cyclic}, we introduce partial cyclic orders and use them to rephrase the results of the previous section. In \cref{sec:PT}, we define the weight function of a positroid polytope using Parke-Taylor functions; prove that the weight function of each tile $\Gamma_{\sigma}$ of $\Delta_{k+1,n}$ is constant; and deduce the Magic Number Theorem.  In \cref{sec:combtiling}, we extend the notion of weight function to facets of tiles and prove additional necessary conditions for when a collection of bicolored subdivsions gives rise to a positroid tiling of $\Delta_{k+1,n}.$  In \cref{sec:combPT}, we introduce Parke--Taylor polytopes and prove new relations among Parke--Taylor functions.   
Finally in \cref{sec:G} we give an interpretation of our previous results in terms of the $G$-amplituhedron.
\bigskip

\noindent{\bf Acknowledgements:~} 
This paper is an offshoot from a collaboration between the authors, Chaim Even-Zohar, and Tsviqa Lakrec; the authors are grateful for many informative discussions with Even-Zohar and Lakrec on the amplituhedron.  The authors are also grateful to Nima Arkani-Hamed for inspiring conversations.
MP is supported by the CMSA at Harvard University and at the Institute for Advanced Study by the U.S. Department of Energy under the grant number DE-SC0009988. MP would like to thank Nick Early for useful conversations. 
MSB is supported by the National Science Foundation under Award No.~DMS-2103282.
RT is supported by the ISF grants No.~335/19 and 1729/23.
LW would like to thank her good friends the Park-Taylor family, who have no relation to the Parke-Taylor identity.  She is supported by the National Science Foundation under Award No. 
 DMS-2152991. Any opinions, findings, and conclusions or recommendations expressed in this material are
those of the author(s) and do not necessarily reflect the views of the National Science
Foundation.

\section{Background: the hypersimplex, the $m=2$ amplituhedron, and their tilings}\label{sec:background}

In this section, we review background on the positive Grassmannian, the hypersimplex, and the amplituhedron. We introduce the notion of positroid tiles and tilings of a space $X$ which is the surjective image of the positive Grassmannian. We review the characterization of positroid tiles for the hypersimplex \cite{LPW} and for the $m=2$ amplituhedron \cite{PSW} using the combinatorics of bicolored subdivisions; we also state the T-duality theorem of \cite{PSW} which connects positroid tilings of the hypersimplex and $m=2$ amplituhedron. Finally, we discuss a triangulation of the hypersimplex into $w$-simplices \cite{StanleyTriangulation,SturmfelsGrobner,LamPost}, which is the simultenous refinement of all positroid tilings and is essential to the proofs in later sections.

\subsection{The (positive) Grassmannian}\label{sec:posGrass}

The \emph{Grassmannian} $\Gr_{k,n} = \Gr_{k,n}(\mathbb{R})$
is the space of all $k$-dimensional subspaces of
an $n$-dimensional vector space $\mathbb{R}^n$.
Let $[n]$ denote $\{1,\dots,n\}$, and $\binom{[n]}{k}$ denote the set of all $k$-element 
subsets of $[n]$.
We can
 represent a point $V \in
\Gr_{k,n}$  as the row-span 
of
a full-rank $k\times n$ matrix $C$ with entries in
$\mathbb{R}$.  
Then for $I=\{i_1 < \dots < i_k\} \in \binom{[n]}{k}$, the \emph{Pl\"ucker coordinate} $P_I(C)$ be the $k\times k$ minor of $C$ using the columns $I$. 
The {Pl\"{u}cker coordinates} of $V$ are independent of the choice of matrix
representative $C$ (up to common rescaling). The \emph{Pl\"ucker embedding} 
$V \mapsto \{P_I(C)\}_{I\in \binom{[n]}{k}}$
embeds  $\Gr_{k,n}$ into
projective space\footnote{
We will sometimes abuse notation and identify $C$ with its row-span.}. 

\begin{definition}[Positive Grassmannian]\label{def:positroid}\cite{lusztig, postnikov}
We say that $V\in \Gr_{k,n}$ is \emph{totally nonnegative}
     if (up to a global change of sign)
       $P_I(C) \geq 0$ for all $I \in \binom{[n]}{k}$.
Similarly, $V$ is \emph{totally positive} if $P_I(C) >0$ for all $I
      \in \binom{[n]}{k}$.
We let $\Grk$ and $\Gr_{k,n}^{>0}$ denote the set of
totally nonnegative and totally positive elements of $\Gr_{k,n}$, respectively.  
The set $\Grk$ is called the \emph{totally nonnegative}  \emph{Grassmannian}, or
       sometimes just the \emph{positive Grassmannian}.
\end{definition}

If we partition $\Grk$ into strata based on which Pl\"ucker coordinates are strictly
positive and which are $0$, we obtain a cell decomposition of $\Grk$
into \emph{positroid cells} \cite{postnikov} that glue together to form a CW complex \cite{PSWv1}.
Each positroid cell $S$ gives rise to a matroid $\mathcal{M}$, whose bases are precisely
the $k$-element subsets $I$ such that the Pl\"ucker coordinate
$P_I$ does not vanish on $S$; the matroid $\mathcal{M}$ 
is called a \emph{positroid}.

\subsection{Positroid tiles and tilings}
Given any surjective map $\phi$ from a cell complex onto a topological space $X$, say of dimension $d$, it is natural to try to decompose $X$ using images
of $d$-dimensional cells under $\phi$.   
This leads to the definition of \emph{$\phi$-induced tiling} \cite[Definition 2.19]{WilliamsICM} which we apply
in the case that our cell complex is the positive Grassmannian.\footnote{There are many reasonable
variations of this definition.  One might want to relax the injectivity
assumption, or to impose further restrictions on how boundaries of the 
images of cells should overlap. Note that in the literature, positroid tilings are sometimes
called \emph{(positroid) triangulations}. We avoid this terminology in order to avoid confusion with the notion of e.g. polytopal triangulations.}  
When $\phi$ is the amplituhedron map $\tilde{Z}$,
tilings were first studied in 
\cite{arkani-hamed_trnka} (where they were referred to as \emph{triangulations}).

\begin{definition}\label{def:tiling0} 
	Let $\phi: \Grk \to X$ be 
 a continuous surjective map where $\dim X=d$. The image ${\phi(S)}$ of a positroid cell $S \subset \Grk$ is an \emph{open tile} for $X$ (with respect to $\phi$) if $\phi$ is injective on $S$ and $\dim S =d$. The closure $\overline{\phi(S)}$ of an open tile is called a \emph{tile} for $X$.
 
A \emph{positroid tiling} of $X$ (with respect to $\phi$) is a collection $\{\overline{\phi(S)}\}_{S \in \mathcal{C}}$ of tiles 
satisfying:
\begin{itemize}
	\item (disjointness): open tiles $\phi(S)$
		and $\phi(S')$ for distinct positroid cells $S$ and $S'$ in $\mathcal{C}$ are disjoint
              \item (covering): $\cup_{S \in \mathcal{C}}\overline{\phi(S)} = X$.
	\end{itemize}
\end{definition}

We will focus here on $\phi$-induced tilings for two different maps $\phi$: the moment map $\mu$ and the amplituhedron map $\tZ$, which map $\Grk$ onto the hypersimplex and the amplituhedron, respectively. In the subsequent subsections, we will discuss the combinatorics of tiles and tilings for these choices of $\phi$, which turn out to be closely related.

\subsection{The moment map and the hypersimplex}\label{subsec:HYS}
We use $\{e_1, \dots, e_n\}$ to denote the standard basis of $\R$, and define $e_I := \sum_{i \in I} e_i \in \R^n$. If $x \in \R^n$ and $I \subset [n]$, we use the notation $x_I:=\sum_{i \in I} x_i$. For $i \neq j \in [n]$, we use the notation $[i,j]:=\{i, i+1, \dots, j-1, j\}$ for the cyclic interval from $i$ to $j$.

The \emph{moment map} from the Grassmannian $\Gr_{k+1,n}$ to $\R^n$
is defined as follows.
\begin{definition}[The moment map] \label{def:mom_map}
	Let $A$ be a $(k+1) \times n$ matrix representing a point of 
	$\Gr_{k+1,n}$.
	The \emph{moment map} 
	$\mu: \Gr_{k+1,n} \to \R^n$ is defined by 
	$$\mu(A) = \frac{ \sum_{I \in \binom{[n]}{k+1}} |p_I(A)|^2 e_I}
	{\sum_{I \in \binom{[n]}{k+1}} |p_I(A)|^2}.$$
\end{definition}

It is well-known that the image of the Grassmannian $\Gr_{k+1,n}$ under the moment
map is the \emph{hypersimplex}
$\Delta_{k+1,n}$, defined below. 
If one restricts the moment map to $\Gr_{k+1,n}^{\geq 0}$ then the image is again the hypersimplex
$\Delta_{k+1,n}$ 
\cite[Proposition 7.10]{tsukerman_williams}.

\begin{definition}[The hypersimplex]
The \emph{$(k+1,n)$-hypersimplex} 
$\Delta_{k+1,n} \subset \R^n$
 is the convex hull of the points $e_I$ where $I$ runs over $\binom{[n]}{k+1}$. 
\end{definition}

See \cref{fig:HYS24} for an example of the hypersimplex. The hypersimplex $\Delta_{k+1,n}$ is $(n-1)$-dimensional, as it is contained in the hyperplane $x_{[n]}=k+1$.

\begin{remark}\label{rem:hypercube}
Under the projection $\pi:(x_1,\dots,x_n) \mapsto (x_1,\dots,x_{n-1})$,
we see that the hypersimplex is linearly equivalent to the polytope in
$\R^{n-1}$ defined by
$$\{(x_1,\dots,x_{n-1}) \ \vert \
0 \leq x_i \leq 1; k \leq x_{[n-1]} \leq k+1\}.$$
That is, the projected hypersimplex
$\pi(\Delta_{k+1,n})$ is the slice of the
unit hypercube $\mbox{\mancube}_{n-1}$ contained between the hyperplanes
$x_{[n-1]} = k$ and $x_{[n-1]}= k+1$.
\end{remark} 

\begin{figure}[h]
\includegraphics[width=0.22\textwidth]{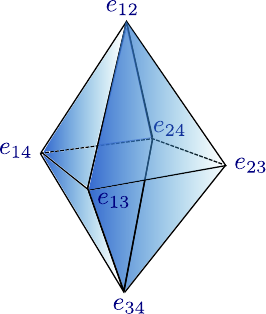}
	\caption{The hypersimplex $\Delta_{2,4}$ projected in $\mathbb{R}^3$.}
	\label{fig:HYS24}
\end{figure}

The moment map images of positroid cells, and their closures, can easily be described using the combinatorics of positroids.

\begin{definition}\label{def:positroid-polytope}
    Let $\mathcal{M} \subset \binom{[n]}{k+1}$ be a positroid. The \emph{positroid polytope} $\Gamma_{\mathcal{M}}$ is defined as 
    \[\Gamma_{\mathcal{M}} := \mbox{conv}\{e_I : I \in \mathcal{M}\}.\]
    We denote by $\Gamma_{\mathcal{M}}^\circ$ the relative interior of $\Gamma_{\mathcal{M}}$.
\end{definition}

\begin{proposition}[{\cite[Proposition 7.10]{tsukerman_williams}}]
    Let $S \subset \Gr_{k+1, n}^{\ge 0}$ be a positroid cell, and $\mathcal{M}$ its corresponding positroid. Then 
    \[\mu(S)= \htile{\mathcal{M}}^\circ \qquad \text{ and } \qquad \overline{\mu(S)}=\htile{\mathcal{M}}. \]
\end{proposition}

For the reader's convenience, below we specialize \cref{def:tiling0} with $\phi$ being the moment map $\mu$ and $X$ being the hypersimplex $\Delta_{k+1,n}$.
\begin{definition} Let $\mu: \Gr^{\geq 0}_{k+1,n} \to \Delta_{k+1,n}$ be the moment map.
 The positroid polytope $\htile{\mathcal{M}} =\overline{\mu(S)}$, which is the closure of the image of a positroid cell $S$, is a \emph{tile} for $\Delta_{k+1,n}$ if $\mu$ is injective on $S$ and $\dim S=n-1$.
 A \emph{positroid tiling} of $\Delta_{k+1,n}$ is a collection of tiles with disjoint interiors whose union is $\Delta_{k+1, n}$.
\end{definition}

\subsection{Tiles of $\Delta_{k+1,n}$ from bicolored subdivisions}

In this section we review the characterization
of tiles of $\Delta_{k+1,n}$ from \emph{bicolored subdivisions}. Here, we will describe the tiles directly as subsets of $\R^n$; see \cite[Section 3]{LPW} for a discussion of the cells whose moment map images are the tiles. We first review the necessary combinatorics of bicolored subdivisions.

\begin{figure}[h]
\includegraphics[width=0.2\textwidth]{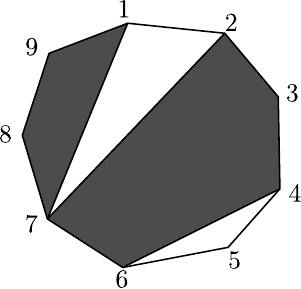}
        \caption{A bicolored subdivision of type $(5,9)$.}
        \label{fig:unpunctured}
\end{figure}

\begin{definition}[Bicolored subdivisions and arcs]\label{def:bicolored} 
Let  $\mathbf{P}_n$ be a convex $n$-gon with vertices labeled from
$1$ to $n$ in clockwise order.
       A \emph{bicolored subdivision $\sdiv$} is a partition of 
       $\mathbf{P}_n$ into black and white polygons 
       such that two polygons sharing an edge have different colors.  We say that $\sdiv$ has 
       \emph{type $(k,n)$} or that \emph{$\sdiv$ is a $(k,n)$-bicolored subdivision} if any triangulation of the black polygons  
       consists of exactly $k$ black triangles.  See \cref{fig:unpunctured}.

       Given a pair of vertices $i,j$ of $\mathbf{P}_n$, we say that the arc $i\to j$ is:
       \begin{itemize}
           \item \emph{compatible} with $\sigma$ if the arc either bounds a polygon 
           or lies entirely inside a polygon of $\sigma$;
           \item \emph{facet-defining} if it bounds a black polygon of 
           $\sigma$ to its left;
           \item \emph{internal} if it is not a boundary edge of $\mathbf{P}_n$. 
       \end{itemize}
       Suppose an arc $i \to j$ is compatible with $\sdiv$. The \emph{area to the left of }$i \to j$, denoted by  
       $\area_{\sdiv}(i \to j)$ or $\area(i \to j)$, is the number of black triangles to the left of $i \to j$ in any triangulation of the black polygons of $\sdiv.$
       
\end{definition}
\begin{example}
For the subdivision $\sdiv$ pictured in \cref{fig:unpunctured}, $3 \rightarrow 7$, $6 \rightarrow 4$, $1 \rightarrow 2$ are compatible, whereas $1 \rightarrow 4$ is not; $6 \rightarrow 4$, $2 \rightarrow 7$, $4 \rightarrow 3$ are facet-defining arcs, whereas $1 \rightarrow 2$ is not facet-defining. Moreover, $\area (3 \rightarrow 7)=2$, $\area (6 \rightarrow 4)=5$, and $\area (1 \rightarrow 2)=0$.
\end{example}

\begin{remark}
 Bicolored subdivisions of type $(k,n)$ are in bijection with \emph{separable permutations} on $[n-1]$ with $k$ descents \cite[Corollary 12.6]{PSW}, which are enumerated by a refinement of the \emph{large Schr\"oder numbers} \cite[A175124]{OEIS}.   
\end{remark}

The following characterization of tiles of $\Delta_{k+1,n}$ in terms of bicolored subdivisions of type $(k,n)$ appears in \cite{PSW} as Proposition 9.5 and 9.6; the characterization of tiles of 
$\Delta_{k+1,n}$ in terms of tree plabic graphs (which are dual to bicolored subdivisions) had previously appeared in \cite[Propositions 3.15, 3.16]{LPW}, confirming conjectures of \cite{Lukowski:2019sxw}.
\begin{theorem}\label{thm:hypersimplex-tiles-ineq-facets}
The positroid tiles for $\Delta_{k+1,n}$ are in bijection with bicolored subdivisions of type $(k,n)$. Moreover, if $\sigma$ is such a bicolored subdivision, then the corresponding tile $\Gamma_\sigma$ is cut out of $\mathbb{R}^n$ by the equality $x_{[n]}=k+1$ and by the following inequalities: 
\begin{equation*}
 \area(i\to j) \leq x_{[i,j-1]}\leq \area(i\to j)+1 \qquad \text{ for any compatible arc }i\to j \text{ of }\sigma 
\end{equation*}
or, alternatively, by the following facet inequalities:
\begin{enumerate}
    \item $x_i \geq 0$,  if there is a white polygon of $\sigma$ with vertex $i$;
    \item $x_{[i,j-1]}\geq \area(i\to j)$, if $i \rightarrow j$ is a facet-defining arc of $\sigma$. 
\end{enumerate}

\end{theorem}
\begin{example}
    For $\sigma$ in \cref{fig:unpunctured}, if $x \in \Gamma_{\sigma}$ then $2 \leq x_{[3,6]} \leq 3$ ($3 \rightarrow 7$ is a compatible arc).
    Some facet inequalities of $\Gamma_{\sigma}$ are: $x_5 \geq 0$ ($5$ is a vertex of a white polygon of $\sigma$), $x_{[6,3]}\geq 5$ ($6 \rightarrow 4$ is a facet-defining arc).
\end{example}

\subsection{The amplituhedron}

Building on \cite{abcgpt,hodges},
Arkani-Hamed and Trnka
\cite{arkani-hamed_trnka}
introduced
the \emph{(tree) amplituhedron}, which they defined as 
the image of the positive Grassmannian under a positive linear map.
Let $\Mat_{n,p}^{>0}$ denote the set of $n\times p$ matrices whose maximal minors
are positive.

\begin{definition}[Amplituhedron]\label{definition_amplituhedron}
Let $Z\in \Mat_{n,k+m}^{>0}$, where $k+m \leq n$.
    The \emph{amplituhedron map}
$\tilde{Z}:\Gr_{k,n}^{\ge 0} \to \Gr_{k,k+m}$
        is defined by
        $\tilde{Z}(C):=CZ$,
    where
 $C$ is a $k \times n$ matrix representing an element of
        $\Gr_{k,n}^{\ge 0}$,  and  $CZ$ is a $k \times (k+m)$ matrix representing an
         element of $\Gr_{k,k+m}$.
        The \emph{amplituhedron} $\mathcal{A}^Z_{n,k,m} \subset \Gr_{k,k+m}$ is the image
$\tilde{Z}(\Gr_{k,n}^{\ge 0})$.
\end{definition}

It is convenient to use the following coordinates when discussing points in $\AA$. 
\begin{definition}[Twistor coordinates]\label{def:tw_coords}
Fix
  $Z \in \Mat_{n,k+m}^{>0}$
   with rows  $Z_1,\dots, Z_n \in \R^{k+m}$.
   Given  $Y \in \Gr_{k,k+m}$
   with rows $y_1,\dots,y_k$,
   and $\{i_1,\dots, i_m\} \subset [n]$,
   we define the \emph{twistor coordinate} $
	\llrr{{Yi_1} {i_2} \cdots {i_m}}$
        to be  the determinant of the
        matrix with rows
        $y_1, \dots,y_k, Z_{i_1}, \dots, Z_{i_m}$.
\end{definition}
\noindent Note that the twistor coordinates are defined only up to a common scalar multiple.

Specializing \cref{def:tiling0}, with $\phi$ being the amplituhedron map $\tilde{Z}$ and $X$ being the amplituhedron $\AA$ we have the following.

\begin{definition} Let $\tilde{Z}: \Gr^{\geq 0}_{k,n} \to \Gr_{k,k+m}$ be the amplituhedron map.
 The image $\gto{S} :={\tilde{Z}(S)}$ of a positroid cell $S$ is an \emph{open tile} for $\AA$ if $\tilde{Z}$ is injective on $S$ and $\dim S =mk$. The closure $\atile{S} :=\overline{\tilde{Z}(S)}$ of an open tile is a \emph{tile} for $\AA$.
 A \emph{positroid tiling} of $\AA$ is a collection $\{\atile{S}\}_{S \in \mathcal{C}}$ of tiles whose union is $\AA$ and such that the open tiles $\{\gto{S}\}_{S \in \mathcal{C}}$ are pairwise disjoint.
\end{definition}

In this paper it is convenient to work with ``all-$Z$'' tilings, as defined below.

\begin{definition}[All-$Z$ tilings]\label{def:tiling}
We call $\{Z_{S} \}_{S\in \mathcal{C}}$ an \emph{all-$Z$ tiling} of the amplituhedron $\mathcal{A}_{n,k,m}$ if 
		$\{Z_{S} \}_{S\in \mathcal{C}}$ is a tiling of $\mathcal{A}_{n,k,m}(Z)$
  for all  $Z\in \Mat_{n,k+m}^{>0}$.
\end{definition}

\begin{remark}
In this paper we will restrict our attention to all-$Z$ tilings\footnote{Currently, there is no known example of tiling which is not an all-Z tiling.}.  
However, for ease of notation, we often drop the phrase ``all-$Z$.''  When we refer to tilings $\{Z_S\}$ we will always mean all-$Z$ tilings.
\end{remark}

There is a natural notion of  \emph{facet} of a tile, 
generalizing the notion of facet of a polytope. 

\begin{definition}[Facet of a cell and a tile]\label{def:facet2} 
Given two positroid cells $S'$ and $S$, we say that 
$S'$ is a \emph{facet} of $S$ if 
$S' \subset \partial{S}$ and $S'$ has codimension $1$ in $\overline{S}$.
If $S'$ is a facet of $S$ and $Z_S$ is a tile of $\AA$, we say that $Z_{S'}$ is a
	\emph{facet} 
of $Z_{S}$ if 
      $\gt{S'} \subset \partial \gt{S}$ and has codimension 1 in $\gt{S}$.
\end{definition}

In this paper, we will be concerned with the amplituhedron when $m=2$. All subsequent sections are in this setting.

\subsection{Tiles of $\mathcal{A}_{n,k,2}$ from bicolored subdivisions}

In \cite[Theorem 4.25]{PSW}, a subset of the authors of the present work 
showed that positroid tiles for 
$\mathcal{A}_{n,k,2}(Z)$ are in bijection with the bicolored subdivisions of type
$(k,n),$ by associating a plabic graph to each bicolored subdivision and showing 
that these plabic graphs exactly correspond to the $2k$-dimensional cells on 
which the amplituhedron map is injective.  Moreover, \cite[Theorem 4.28]{PSW}
gives a concise description of these positroid tiles in terms of inequalities.
These two results are summarized in the following theorem.

\begin{theorem}\label{thm:amp-tiles-facets-inequalities}
The positroid tiles for $\mathcal{A}_{n,k,2}(Z)$ are in bijection with the bicolored subdivisions of type $(k,n).$  Moreover, if $\sigma$ is such a bicolored subdivision, then the corresponding open tile is
$$Z^{\circ}_\sigma = \{Y\in \Gr_{k,k+2} \ \vert \ 
 (-1)^{\area(i\to j)}\llrr{Y i j} >0 \text{ for all compatible arcs }i\to j \text{ of }\sigma \text{ with }i<j\}.$$
Alternatively, let $\tilde{\sigma}$ be a triangulation of the black polygons of $\sigma$. Then 
\begin{enumerate}
    \item $Z^{\circ}_{\sigma}$ is cut out by the inequalities $(-1)^{\area(i\to j)}\llrr{Y i j} >0$ for all arcs $i\to j$ in $\tilde{\sigma}$ with $i<j$. 
    \item the facets of $Z_\sigma$ lie on the hypersurfaces $\llrr{Y i j} = 0$, where $i\to j$ a facet-defining arc of $\sigma$. 
\end{enumerate}
    \end{theorem}
\begin{example}
    For $\sigma$ in \cref{fig:unpunctured}: if $Y \in Z^{\circ}_{\sigma}$ then
    $(-1)^{2}\llrr{Y 3 7} >0$ ($3 \rightarrow 7$ is a compatible arc).
    There is a facet of $Z_{\sigma}$ lying on $\llrr{Y 6 4}=0$ ($6 \rightarrow 4$ is a facet-defining arc).
\end{example}

\subsection{Correspondence between the hypersimplex and the amplituhedron}
From \cref{thm:hypersimplex-tiles-ineq-facets,thm:amp-tiles-facets-inequalities}, we see that tiles of $\Delta_{k+1, n}$ and $\Ank$ are both in bijection with bicolored subdivisions of type $(k,n)$, and thus in bijection with each other. In fact, the bijection between tiles of the hypersimplex 
$\Delta_{k+1,n}$ and tiles of the amplituhedron
$\mathcal{A}_{n,k,2}$ induces a bijection on tilings as well. 
\begin{theorem}[{\cite[Theorem 11.6]{PSW}}]\label{thm:Tduality}
Let $\mathcal{S}$ be a collection of bicolored subdivisions of type $(k,n)$. Then
$\{\htile{\sdiv} \}_{\sdiv\in \mathcal{S}}$ is a positroid tiling of $\Delta_{k+1,n}$ if and only if $\{\atile{\sdiv} \}_{\sdiv\in \mathcal{S}}$ is an all-$Z$ tiling of $\A_{n, k, 2}$.    
\end{theorem}

There is a particularly nice class of tilings which we can now describe in terms of 
bicolored subdivisions.

\begin{definition}[Bicolored subdivisions of kermit type]\label{def:kermit}
Let $v\in [n]$, let $I = \{i_1,\dots,i_k\}\in {[n]\setminus \{v\} \choose k}$ and let 
$\sigma_I$ be the bicolored subdivision of the $n$-gon obtained by drawing 
black triangles with vertices $\{v, i_\ell, i_{\ell}+1\}$ for $\ell=1,\dots,k$. 
(We then repeatedly merge black triangles sharing a common edge into a larger
black polygon.)  We say that $\sigma^v_I$ is a \emph{kermit bicolored subdivision based at vertex $v$}. We will often denote $\sigma^1_I$ by $\sigma_I$.
\end{definition}

\begin{figure}[h]
\includegraphics[width=0.8\textwidth]{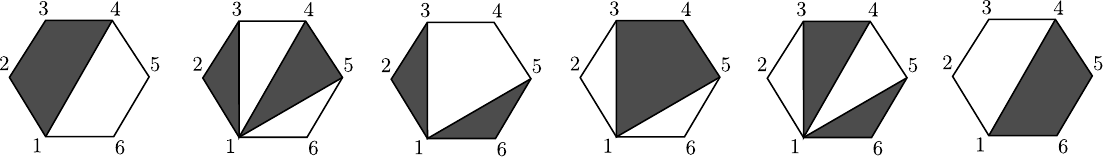}
        \caption{The bicolored subdivisions (based at vertex $v=1$) giving rise to the kermit tilings of 
        $\Delta_{3,6}$ and of $\mathcal{A}_{6,2,2}$.}
        \label{fig:tilings}
\end{figure}

\begin{proposition}\cite[Proposition 11.32]{PSW}\label{prop:tilings}
Choose $0 \leq k \leq n-2$ and let $I$ run over ${[2,n-1] \choose k}$.  Then the collections
$\{\Gamma_{\sigma_I}\}$ and $\{Z_{\sigma_I}\}$ are positroid tilings of 
$\Delta_{k+1,n}$ and $\mathcal{A}_{n,k,2},$ respectively.  In particular,
$$\Delta_{k+1,n} = \bigcup_I \Gamma_{\sigma_I} \text{ and }
\mathcal{A}_{n,k,2}(Z) = \bigcup_I Z_{\sigma_I}.$$
Moreover if we apply a cyclic shift to our definitions
(i.e. rotate the labels of the $n$-gon), we obtain
positroid tilings of $\Delta_{k+1,n}$ and $\mathcal{A}_{n,k,2}$.
\end{proposition}

For any $v\in [n]$, we refer to the tiling $\{\Gamma_{\sigma_I^v}\}$ 
of $\Delta_{k+1,n}$, and the tiling $\{Z_{\sigma_I^v}\}$ 
of $\mathcal{A}_{n,k,2},$ as \emph{kermit tilings (based at $v$)}, see \cref{fig:tilings}.

See \cite[Table 1]{PSW} for many other correspondences between the hypersimplex and the $m=2$ amplituhedron.

\subsection{The triangulation of the hypersimplex into $w$-simplices} 
In this section, we review a unimodular triangulation of the hypersimplex, originally due to Stanley \cite{StanleyTriangulation}, and also studied by Sturmfels \cite{SturmfelsGrobner} and Lam--Postnikov \cite{LamPost}. The maximal simplices of this triangulation are indexed by permutations with a fixed number of descents, which are counted by the Eulerian numbers. We follow the presentation of \cite{LamPost}.

\begin{definition} 
Let $w \in S_n$. A letter $i<n$ is a \emph{left descent}
 (or  a \emph{left descent bottom}) of $w$ if $i$ occurs to the right of $i+1$ in $w$.  In other words,
 $w^{-1}(i) > w^{-1}(i+1)$.
 And we say that $i\in [n]$ is a \emph{cyclic left descent} of $w$ if either
 $i<n$ is a left descent of $w$ or if $i=n$ and $1$ occurs to the left of $n$ in $w$, 
 that is, $w^{-1}(1) < w^{-1}(n)$.
We let $\cdes(w)$ denote the set of cyclic left descents of $w$. 
 We frequently refer to cyclic left descents as simply \emph{cyclic descents}.
\end{definition}

\begin{remark}
Left and right descents and descent sets are discussed extensively 
in \cite[Chapter 1]{bjornerbrenti}.
Left descents are sometimes called \emph{recoils} in the literature.
\end{remark}

\begin{definition}\label{Dk+1n}
Choose $0 \leq k \leq n-2$. We define $D_{k+1, n}$ to be the set of permutations $w \in S_n$ with $k+1$ cyclic descents and $w_n=n$. We let $D_n$ be the set of permutations $w \in S_n$ with $w_n=n$. 
\end{definition}
Note that $|D_{k+1, n}|$ equals the \emph{Eulerian number} $E_{k,n-1}:=\sum_{\ell=0}^{k+1}(-1)^{\ell} {n \choose \ell}(k+1-\ell)^{n-1}$. 

\begin{example}\label{ex:D24}
  The set $D_{2,4}$ contains the following $E_{2,3}=4$ permutations.
  
  \begin{center}
      \begin{tabular}{|c|c|c|c|c|}
        \hline  $w$ & 1324 & 3124 & 2134 & 2314 \\
        \hline  $\cdes(w)$ & $\{2,4\}$ & $\{2,4\}$ & $\{1,4\}$ & $\{1,4\}$ \\ \hline
      \end{tabular}
  \end{center}
\end{example}

It will often be convenient for us to identify elements of $D_n$ with $n$-cycles in $S_n$. In \cref{sec:cyclic}, we will also identify elements of $D_n$ and $n$-cycles with total cyclic orders.

\begin{notation}\label{not:long-cycle}
    For $w=w_1 w_2 \dots w_n$ in $S_n$, we write $(w)=(w_1 \ w_2 \cdots w_n)$ for the cycle which sends $w_1 \mapsto w_2 \mapsto \cdots \mapsto w_n \mapsto w_1$. Note that $w \mapsto (w)$ gives a bijection between $D_n$ and the set of $n$-cycles in $S_n$.
\end{notation}

\begin{definition}[$w$-simplices] \label{definition:wsimplexHSimplex}
 For $w=w_1 w_2 \dots w_n \in D_{k+1, n}$, let $w^{(a)}$ denote the cyclic rotation of $w$ ending at $a$.
	We define 
	\[I_r=I_r(w):=\cdes(w^{(r)}).
	\]
 Note that $I_r$ in fact depends only on $(w)$ rather than $w$ itself. We define the \emph{w-simplex} $\simp{w} \subseteq \Delta_{k+1, n}$ to be the convex hull of the points $e_{I_1}, \dots, e_{I_n}$; this is an $(n-1)$-dimensional simplex. We call $$I_{w_1} \to I_{w_2} \to \dots \to I_{w_n} \to I_{w_1}$$ the \emph{circuit} of $\simp{w}$ and sometimes abuse notation by equating $\simp{w}$ and its circuit.
\end{definition}

\begin{example}\label{ex:w_simplex}
    From \cref{ex:D24}, there are four $w$-simplices in $\Delta_{2,4}$, see \cref{fig:HYSwsimp}. For $w= 1324$, to determine the vertices of $\simp{w}$, we compute the cyclic descents of the rotations of $w$.

\begin{center}
    \begin{tabular}{|c|c|c|c|c|}\hline
     $r$ & 1 & 2 & 3 & 4 \\ \hline
       $w^{(r)}$ & 3241 & 4132 & 2413 & 1324 \\ \hline
      $I_r=\cdes(w^{(r)})$ & $\{1,2\}$  & $\{2,3\}$ & $\{1,3\}$ & \{2,4\} \\ \hline
    \end{tabular}
\end{center}

    Therefore, $\Delta_{(1324)}$ is the convex hull of  $e_{12},e_{23},e_{13},e_{24}$, i.e. it is the red simplex in \cref{fig:HYSwsimp}. 

    \begin{figure}[h]
\includegraphics[width=0.22\textwidth]{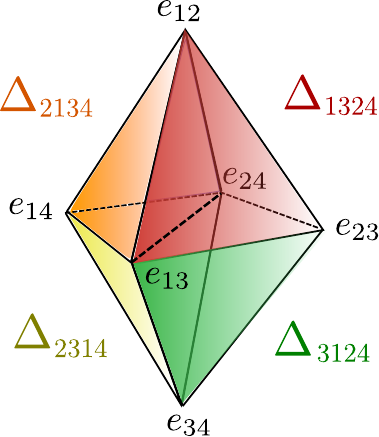}
	\caption{$w$-simplices for $\Delta_{2,4}$. We omit parentheses in subscripts for readability.}
	\label{fig:HYSwsimp}
\end{figure}
\end{example}

\begin{remark}\label{rmk:convention-change}
\cite[Section 10]{PSW} uses left cyclic descent tops rather than left cyclic descent bottoms. It also uses the notation $\Delta_w$ for $w$-simplices. The bijection $w \mapsto \simp{w}$ from $D_{k+1, n}$ to $w$-simplices in \cref{definition:wsimplexHSimplex} is slightly different from the bijection $u \mapsto \Delta_{u}$ of \cite[Definition 10.3]{PSW}. In particular, let $u$ be the permutation obtained from $w$ by subtracting $1$ from each number and rotating so $n$ is at the end. Then $\simp{w}=\Delta_{u}$. The left cyclic descent tops of $u^{(r-1)}$ are the same as the left cyclic descent bottoms of $w^{(r)}$, so the sets $I_r(w)$ defined here are the same as the sets $I_r(u)$ in \cite[Definition 10.3]{PSW}.
\end{remark}

\begin{remark}\label{remark:circuitdefinition}
For a $w$-simplex
$\simp{w}$,
the consecutive elements $I_{w_{i-1}} \to I_{w_{i}}$ of the circuit are related by 
$$I_{w_{i}} = I_{w_{i-1}} \setminus\{w_i -1\} \cup \{w_i\}. $$
(In particular, $r$ is always in $I_r$ and $r-1$ is never in $I_r$.)
From this, one can see that, letting $u$ be as in \cref{rmk:convention-change},  $\simp{w}$ agrees with the simplex denoted $\Delta_{(u)}$ in \cite[Section 2.4]{LamPost}. In particular, the circuit $I_{w_1} \to I_{w_2} \to \dots \to I_{w_n} \to I_{w_1}$ is the circuit used to define $\Delta_{(u)}$. 
\end{remark}

The following triangulation of the hypersimplex first appeared in \cite{StanleyTriangulation},
though the description there was slightly different.
\begin{proposition}[$\Delta_{k+1,n}$ is the union of $w$-simplices \cite{StanleyTriangulation}]\label{prop:StanleyTriangulation}
	The $w$-simplices $\{\simp{w}: w \in D_{k+1, n}\}$ are the maximal simplices of a triangulation of the 
 hypersimplex $\Delta_{k+1,n}$.
Moreover,
projecting $\{\simp{w}: w\in S_n\}$ into $\mathbb{R}^{n-1}$ (see \cref{rem:hypercube}),
we obtain the maximal simplices in a triangulation of the hypercube $\mbox{\mancube}_{n-1}$ 
which refines the subdivision of the hypercube into hypersimplices. 
\end{proposition}

It follows from \cite[Theorem 2.7]{LamPost} that the triangulation
$$\Delta_{k+1, n}= \bigcup_{w \in D_{k+1, n}} \simp{w}$$
is obtained by cutting the hypersimplex with all facet hyperplanes of all positroid tiles. Another way to say this is that the triangulation into $w$-simplices is the simultaneous refinement of all positroid tilings.
From this, one can conclude the following.

\begin{proposition}[\cite{LamPost}] \label{prop:w-simp-triangulate-pos-poly}
 Every positroid tile\footnote{In fact, this statement hold for any full-dimensional positroid polytope.} for $\Delta_{k+1, n}$ has a triangulation into $w$-simplices.
\end{proposition}

We will later need to discuss facet-sharing of $w$-simplices. The characterization of facet-sharing was given in \cite[Theorem 2.9]{LamPost}. The hyperplane on the shared facet lies can be deduced readily from \cite{LamPost} (we also prove it in \cref{cor:ineq-and-facet-w-simp}).

\begin{proposition}\label{prop:w-simp-facet-sharing}
    Let $u,w \in D_{k+1, n}$. Then $\simp{w}$ and $\simp{u}$ share a facet if and only if $(w)$ and $(u)$ are related by swapping an adjacent pair of numbers, i.e. if $(w)=(A\ i \ j \ B)$ and $(u)=(A \ j \ i \ B)$, where $j \neq i \pm 1$. The shared facet, which we denote by $\simp{w}^{(ij)}$ or $\simp{u}^{(ji)}$, lies on a hyperplane $x_{[i,j-1]}=c$.
\end{proposition}

Finally, the results from this section have analogues in the $m=2$ amplituhedron.

\begin{remark}\label{rmk:w-chambers}
    Cutting $\Delta_{k+1, n}$ with the facet hyperplanes of all positroid tiles decomposes the hypersimplex into $w$-simplices $\simp{w}$. Analogously, cutting $\Ank$ with the facet hypersurfaces of all positroid tiles decomposes the amplituhedron into \emph{$w$-chambers} ${\cham{w}}$ 
    (cf. \cref{def:ampchamber}), which are also indexed by $w \in D_{k+1, n}$ \cite[Corollary 11.20]{PSW}. Each chamber consists of points whose twistor coordinates have specified signs, which can be read off of $w$. Some $w$-chambers may be empty for a particular choice of $Z$, but all are nonempty for some $Z \in \Mat^{>0}_{n, k+2}$. 
\end{remark}

\begin{remark}\label{rmk:wsimp-in-tile-iff-wcham-in-tile}
    The bijection between tiles of the hypersimplex and the amplituhedron can be further refined using $w$-simplices and $w$-chambers. Each tile $\atile{\sdiv}$ of the amplituhedron is a union of $w$-chambers \cite[Corollary 10.17]{PSW}. If $\cham{w}$ is nonempty, then $\cham{w} \subset \atile{\sdiv}$ if and only if $\simp{w} \subset \htile{\sdiv}$ \cite[Proposition 11.1]{PSW}.
\end{remark}

\section{Combinatorial results on tiles}\label{sec:combtiles}

In this section, we characterize which $w$-simplices are contained in a tile $\htile{\sigma}$.

\begin{theorem}\label{thm:when-w-simp-in-tile}
    Let $\sdiv$ be a bicolored subdivision of type $(k,n)$, and $\htile{\sdiv}$ the corresponding tile of $\Delta_{k+1, n}$. Then
    \[\htile{\sdiv}= \bigcup \simp{w}\]
    where the union is over $w \in D_{k+1, n}$ which satisfy
    \begin{itemize}
        \item for every white polygon of $\sdiv$ with vertices $v_1, \dots, v_r$ in \textbf{clockwise} order, we see $v_1, \dots, v_r$ in order in $(w)$.
        \item for every black polygon of $\sdiv$ with vertices $v_1, \dots, v_r$ in \textbf{counterclockwise} order, we see $v_1, \dots, v_r$  in order in $(w)$.
    \end{itemize}
    \end{theorem}

\begin{example}\label{ex:order}
    For the bicolored subdivision $\sdiv$ in \cref{fig:ex_order}, $\Gamma_\sigma$ is the union of all $w$-simplices $\simp{w}$ such that  $w \in D_{3,6}$ satisfies:
    \begin{itemize}
        \item we see $1,2,3$ and $1,4,6$ in order in $(w)$ (from the white polygons);
        \item we see $1,4,3$ and $4,6,5$ in order in $(w)$ (from the black polygons).
        \end{itemize}
\begin{figure}[h]
\includegraphics[height=1.2in]{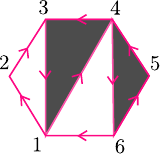}
        \caption{A bicolored subdivision $\sdiv$, together with the orientation of the boundary of the white (resp. black) polygons clockwise (resp. counterclockwise).}
        \label{fig:ex_order}
\end{figure}
        These permutations $w$ are:
        \begin{center}
    \begin{tabular}{|c|c|c|c|c|c|c|c|}\hline
   $5 1 2 4 3 6$ & $3 5 1 2 4 6$ & $5 3 1 2 4 6$ & $1 2 5 4 3 6$ & $3 1 2 5 4 6$ & $1 5 2 
4 3 6$ & $3 1 5 2 4 6$ & $1 5 4 236$  \\ \hline
      $3 1 5 4 2 6$ & $2 3 1 5 4 6$ & $514236$ & $351426$ & $531426$ & $235146$ & $253146$ & $523146$  \\ \hline
    \end{tabular}
\end{center}
For example, $512436$ satisfies the conditions above: $5 \fbox{12} 4 \fbox{3} 6$, $5 \fbox{1} 2 \fbox{4} 3 \fbox{6}$, $5 \fbox{1} 2 \fbox{43} 6$, $\fbox{5} 1 2 \fbox{4} 3 \fbox{6}$.
\end{example}

    Before proving \cref{thm:when-w-simp-in-tile}, we need a straightforward lemma.

\begin{lemma}\label{lem:w-simp-interval-intersection}
	Consider $w \in D_n$ with $(w)= (\cdots q \ a \cdots r \ b \cdots)$ and corresponding $w$-simplex
	\[\simp{w}=\cdots I_q\to I_a \to \cdots \to I_r \to I_b \to \cdots.\]
	If $j$ appears weakly between $a$ and $r$ in $(w)$, then
	\[|I_a \cap [a, b-1]| = |I_j \cap [a, b-1]| .\]
	If $j$ appears weakly between $b$ and $q$ in $(w)$, then 
	\[|I_a \cap [a, b-1]|-1=|I_b \cap [a, b-1]|= |I_j \cap [a, b-1]| .\]
\end{lemma}

    \begin{proof}
First, suppose $j$ appears weakly between $a$ and $r$. If $j=a$, the lemma clearly holds. If $j \neq a$, either $j-1, j$ are both in $[a, b-1]$ or neither $j-1, j$ is in $[a, b-1]$. This is because $j \neq b$ by assumption. Now, the number $x$ preceding $j$ in $(w)$ is also weakly between $a$ and $r$. By definition, $I_j = I_x \setminus \{j-1\} \cup \{j\}$, so we have 
       \[|I_x \cap [a, b-1]|= |I_j \cap [a, b-1]|.\]
       This shows that $|I_j \cap [a, b-1]|$ is constant as $j$ ranges over all numbers appearing weakly between $a$ and $r$ in $(w)$.
       
       We now turn to the second case. An identical argument shows that $|I_j \cap [a, b-1]|$ is constant as $j$ ranges over numbers weakly between $b$ and $q$ in $(w)$. Since $I_b= I_r \setminus \{b-1\} \cup \{b\}$, we conclude that 
       \[|I_b \cap [a, b-1]| =|I_r \cap [a, b-1]|-1 =  |I_a\cap [a, b-1]|-1. \]
       This shows the desired equalities.
    \end{proof}

\cref{lem:w-simp-interval-intersection} implies that $\simp{w}$ satisfies a number of inequalities.

\begin{corollary}\label{cor:ineq-and-facet-w-simp}
	Consider a $w$-simplex $\simp{w} \subset \Delta_{k+1, n}$ and let $I_a$ be as in \cref{definition:wsimplexHSimplex}. Then
	\begin{enumerate}
		\item For all $a, b \in [n]$, all points in $\simp{w}$ satisfy 
		\[|I_b \cap [a, b-1]| \leq x_{[a,b-1]} \leq |I_a \cap [a, b-1]|.\]
		\item A minimal inequality description of $\simp{w}$ as a subset of $\R^n$ is given by the equation $x_{[n]}=k+1$ together with the $n$ facet inequalities:
  \begin{equation*}
     |I_b \cap [a, b-1]| \leq x_{[a, b-1]}
  \end{equation*}
  where $a$ immediately precedes $b$ in $(w)$.
	\end{enumerate}
\end{corollary}

\begin{proof}
\noindent For (1): Say $\simp{w}=\cdots \to I_q\to I_a \to \cdots \to I_r \to I_b \to \cdots$. \cref{lem:w-simp-interval-intersection} shows that all vertices $e_{I_j}$ of $\simp{w}$ lie on one of the hyperplanes
\[x_{[a, b-1]} = |I_a \cap [a, b-1]| \quad \text{or} \quad x_{[a, b-1]} = |I_a \cap [a, b-1]| -1 = |I_b \cap [a, b-1]|.\]
Thus, all points in $\simp{w}$ satisfy
\[|I_b \cap [a, b-1]| \leq x_{[a,b-1]} \leq |I_a \cap [a, b-1]|.\]

\noindent For (2): The $w$-simplex $\simp{w}$ is codimension $1$ in $\R^n$, and its affine span is the hyperplane $x_{[n]}=k+1$.

If $a$ immediately precedes $b$ in $w$, then by \cref{lem:w-simp-interval-intersection}, $n-1$ vertices of $\simp{w}$ lie on the hyperplane
$$x_{[a,b-1]}=|I_b \cap [a,b-1]|.$$ 
These $n-1$ vertices form a facet. This gives rise to $n$ facet hyperplanes of $\simp{w}$, which is the complete list of facet hyperplanes as $\simp{w}$ is a simplex.
\end{proof}

\begin{example}
The $w$-simplex $\Delta_{(1324)}$ from \cref{ex:w_simplex} is the subset of $\R^4$ given by the equation $x_{[4]}=3$ together with the $4$ facet inequalities
  \begin{equation*}
1 \leq x_{[1,2]}, \quad  1 \leq x_{[3,1]}, \quad    1 \leq x_{[2,3]} \quad \text{and} \quad 0 \leq x_4.   
  \end{equation*} 
\end{example}

\begin{proof}[Proof of \cref{thm:when-w-simp-in-tile}]
Because $\htile{\sdiv}$ is triangulated by $w$-simplices (cf. \cref{prop:w-simp-triangulate-pos-poly}), $\simp{w}$ is either contained in $\htile{\sdiv}$ or it does not intersect the interior of $\htile{\sdiv}$. So it suffices to determine which $\simp{w}$ intersect the interior of $\htile{\sdiv}$.

By \cref{thm:hypersimplex-tiles-ineq-facets}, a point in $\R^n$ is in the interior of $\htile{\sdiv}$ if and only if $x_{[n]}=k+1$ and for each arc $a \to b$ compatible with $\sdiv$, the point satisfies
	\[\area(a \to b) < x_{[a, b-1]} < \area(a \to b)+1.\]
	Points in $\simp{w}$ satisfy the equality $x_{[n]}=k+1$ for all $w \in D_{k+1, n}$. By \cref{cor:ineq-and-facet-w-simp}, the above inequalities hold for a point in $\simp{w}$ if and only if 
	\begin{equation}\label{eq:area-equals-intersection}
		\area(a \to b)= |I_b \cap [a, b-1]| \quad \text{ for all arcs } a \to b \text{ compatible with } \sdiv.
	\end{equation}
	
	We now show that \eqref{eq:area-equals-intersection} holds if and only if $w$ satisfies the conditions of the theorem.
	
\noindent $(\implies)$: Suppose that $w$ satisfies the conditions of the theorem. Choose a compatible arc $a \to b$. We proceed by induction on the size of the polygon to the left of the arc $a \to b$. The base case is $b=a+1$, and in this case, 
	\[\area(a \to a+1) =0 = |I_b \cap [a, a]|.\]
	The last equality holds because $I_b$ is obtained from some other subset $I_x$ by removing $b-1=a$ and adding $b$, and so does not contain $a$.
	
	Now, suppose $\area(i \to j)= |I_j \cap [i, j-1]|$ for all arcs which are compatible with $\sdiv$ and have an $(r-1)$-gon to the left. Suppose the arc $a \to b$ has an $r$-gon to its left. Choose some vertex $c$ of this $r$-gon so that $a \to c$ and $c \to b$ are also compatible with $\sdiv$. Suppose the triangle with vertices $a,b,c$ is white; the argument when the triangle is black is very similar. Then 
	\[\area(a \to b)= \area(a \to c) + \area(c \to b)= |I_c \cap [a, c-1]| + |I_b \cap [c, b-1]|\]
where the second equality is by the inductive hypothesis. Note that we see $a,c,b$ clockwise around the boundary of some white polygon in $\sdiv$. So by assumption, we have
\[(w)=(\cdots a \cdots c \cdots b \cdots).\]	
\cref{lem:w-simp-interval-intersection} implies that $|I_b \cap [a, c-1]|= |I_c \cap [a, c-1]|$, so we conclude that 
\[\area(a \to b)= |I_c \cap [a, c-1]| + |I_b \cap [c, b-1]| = |I_b \cap [a, b-1]|\]
as desired.

	\noindent $(\impliedby)$: We show the contrapositive. That is, we show that if $w$ does not satisfy the conditions of the theorem, then \eqref{eq:area-equals-intersection} does not hold.
	
	Suppose the conditions of the theorem fail for some white polygon of $\sdiv$ (the argument for a black polygon is similar). Then there are vertices $a,c, b$ which appear in clockwise order around this white polygon, but we see $(w)=(\cdots a \cdots b \cdots c \cdots)$.
	
	We have $\area(a \to b)= \area(a \to c) + \area(c \to b)$ because the triangle on $a, b,c$ is white. If $\area(i \to j)= |I_j \cap [i, j-1]|$ for all three arcs $a \to b, a \to c$ and $c \to b$, then we would have 
	\[|I_b \cap [a, b-1]|= |I_c \cap [a, c-1]| + |I_b \cap [c, b-1]|.\]
	But by \cref{lem:w-simp-interval-intersection}, $|I_c \cap [a, c-1]|= |I_b \cap [a, c-1]|-1$ and thus 
	\[|I_b \cap [a, b-1]|= |I_c \cap [a, c-1]| + |I_b \cap [c, b-1]|= |I_b \cap [a, b-1]| -1.\]
	This is impossible, so we must have $\area(i \to j)\neq |I_j \cap [i, j-1]|$ for one of $a \to b, a \to c$ or $c \to b$. Since all three arcs are compatible with $\sdiv$, this means \eqref{eq:area-equals-intersection} does not hold.
\end{proof}

In light of \cref{rmk:wsimp-in-tile-iff-wcham-in-tile}, \cref{thm:when-w-simp-in-tile} has the following corollary for the amplituhedron.

 \begin{corollary}\label{cor:when-w-cham-in-tile}
     Let $\sdiv$ be a $(k,n)$-bicolored subdivision and $\atile{\sdiv}$ the corresponding tile of $\Ank$. Then 
     \[\atile{\sdiv}= \bigcup \cham{w}\]
         where the union is over $w \in D_{k+1, n}$ which satisfy
    \begin{itemize}
        \item for every white polygon of $\sdiv$ with vertices $v_1, \dots, v_r$ in \textbf{clockwise} order, we see $v_1, \dots , v_r$  in order in $(w)$.
        \item for every black polygon of $\sdiv$ with vertices $v_1, \dots, v_r$ in \textbf{counterclockwise} order, we see $v_1, \dots , v_r$ in order  in $(w)$.
    \end{itemize}
 \end{corollary}

\section{Cyclic orders and their circular extensions}\label{sec:cyclic}

We now interpret some of the previous results in terms of \emph{(partial) cyclic orders}
and \emph{circular extensions}.   Cyclic orders and circular extensions are circular analogues of partial orders
and linear extensions.

\begin{definition}\label{def:cyclic-order}
    A \emph{(partial) cyclic order} on a finite set $X$ is a ternary relation $C \subset X^3$ such that for all $a, b, c, d \in X$:
    \begin{align}
        (a,b,c) \in C &\implies (c, a, b) \in C \tag{cyclicity}\\
        (a,b,c) \in C &\implies (c, b, a) \notin C \tag{asymmetry}\\
        (a,b,c) \in C \text{ and }(a, c, d) \in C &\implies (a,b,d) \in C \tag{transitivity}
    \end{align}
    A cyclic order $C$ is \emph{total} if for all $a, b, c\in X$, either $(a,b,c)\in C$ or $(a, c, b) \in C$.
  \end{definition}

Informally, a total cyclic order $C$ on $[n]$ is a way of placing $1, \dots, n$ on a circle, just as a total order is a way of placing $1, \dots, n$ on a line.

\begin{definition}
Let $w=w_1\dots w_n \in S_n$. The $w$-order $C_w$ is the total cyclic order obtained by placing $w_1, w_2, \dots, w_n$ on the circle clockwise. We identify this total cyclic order with the $n$-cycle $(w)$ and so may write $(w)$ for $C_w$ or write $C_w=(w_1\ w_2 \dots w_n)$.
\end{definition}  

Note that each total cyclic order on $[n]$ is of the form $C_w$ for a unique permutation $w \in D_n$ (cf. \cref{Dk+1n} for the definition of $D_n$). We move interchangeably between $w \in D_n$, the $n$-cycle $(w)$ and the total cyclic order $C_w$.

\begin{definition}\label{def:circular-extension}
    A total cyclic order $C$ is a \emph{circular extension} of a cyclic order $C'$ if $C' \subset C$. We let $\Ext(C)$ denote the set of all circular extensions of a cyclic order $C$. In an abuse of notation, if $C$ is a partial cyclic order on $[n]$, we sometimes write $(w) \in \Ext(C)$ if $C_w \in \Ext(C)$.
\end{definition}

 Not all cyclic orders have a circular extension \cite{Meg_cyclicOrders}, that is, 
 $\Ext(C)$ could be empty.  Moreover, the problem of determining
whether a cyclic order has a circular extension is NP-complete
\cite{GalilMegiddo}.

\begin{definition}
Let $x_1,\dots,x_m$ be a sequence of $m$ distinct elements of $[n]$ (for $3 \leq m \leq n$).
We let $C=C_{(x_1, x_2 \dots, x_m)}$ denote the partial cyclic order on $[n]$ in which 
for each triple $1\leq i < j < \ell \leq m$ 
        we have $(x_i,x_j, x_{\ell}) \in C$ (which implies by cyclicity that
        also $(x_j,x_{\ell}, x_i)$ and $(x_{\ell}, x_i, x_j)$ lie in $C$).  
We call this partial cyclic order a \emph{chain}.
\end{definition}

We now generalize the notion of bicolored subdivision to \emph{tricolored} (or ``partially bicolored'') subdivision, which will be useful in \cref{sec:combPT}.  We will then
 associate a partial cyclic order to every tricolored subdivision.
\begin{figure}[h]
\includegraphics[height=1.4in]{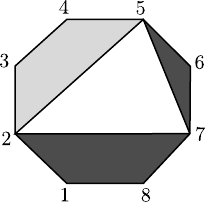}
        \caption{A tricolored subdivision $\tdiv$ of type $(3,2,8)$.  It gives rise to the cyclic order
        $C_{\tdiv}$ which is the union of the chains
        $C_{(2,5,7)}$, $C_{(5,7,6)}$, and $C_{(1,8,7,2)}$.}
        \label{fig:tricolored}
\end{figure}

\begin{definition}[Tricolored subdivisions]\label{def:tricolored}  
Let  $\mathbf{P}_n$ be a convex $n$-gon with vertices labeled from
$1$ to $n$ in clockwise order.
       A \emph{tricolored subdivision $\tdiv$} is a partition of 
       $\mathbf{P}_n$ into black, white, and grey polygons such that two polygons sharing an edge have different colors.  
       We say that $\tdiv$ has 
       \emph{type $(k,\ell, n)$} if any triangulation of the black (respectively, grey) polygons consists of exactly $k$ black (respectively, $\ell$ grey) triangles.
       See \cref{fig:tricolored}.
\end{definition}

\begin{definition}[$\tdiv$-order]\label{def:cyclic-from-perm-subdiv}
Let $\tdiv$ be a tricolored subdivision of $\mathbf{P}_n$ with $q$ polygons $P_1, \dots, P_q$ which are black or white (we ignore the grey polygons).
If $P_a$ is white (respectively, black), we let $v_1,\dots,v_r$ denote its list of vertices read in clockwise (respectively, counterclockwise) order.  We then associate the chain $C_a = C_{(v_1,\dots,v_r)}$ to $P_a$.  Finally we define the \emph{$\tdiv$-order}
to be the partial cyclic order which is the union of the partial cyclic orders associated
to the black and white polygons: 
    \[C_{\tdiv}:=C_1 \cup \dots \cup C_q.\]
\end{definition}
We leave it as an exercise to verify that $C_{\tdiv}$ is a partial cyclic order. See \cref{fig:tricolored} for an example.

To rephrase \cref{thm:when-w-simp-in-tile} and \cref{cor:when-w-cham-in-tile} in terms of cyclic orders, we need one straightforward lemma. Recall the definition of 
$D_{k+1,n}$ and $D_n$ from \cref{Dk+1n}.

\begin{lemma}\label{lem:circular-ext-sdiv-fixed-number-descents}
    Let $\sdiv$ be a bicolored subdivision of type $(k,n)$. Suppose $w \in D_n$ and $C_w$ is a circular extension of $C_\sdiv$. Then $w \in D_{k+1, n}$. 
\end{lemma}

\begin{proof}
We proceed by induction on $n$.
        The base case is $n=3$. For this $n$, $k$ is either 0 or 1. In either case, $C_\sdiv$ is a total cyclic order, and so $C_w =C_\sdiv$. One can check that $w \in D_{k+1, 3}$.

    Now, suppose $n>3$. Without loss of generality, we may assume a polygon $P$ of $\sdiv$ contains the edges of $\mathbf{P}_n$ between $n-1$ and $n$ and between $n$ and 1. 
    Let $\sdiv'$ be the subdivision of $\mathbf{P}_{n-1}$ obtained by removing the triangle on vertices $n-1, n, 1$ from $\sdiv$, and let $w' \in S_{n-1}$ be the permutation obtained from $w^{(n-1)}$ by removing $n$ (recall that $w^{(n-1)}$ denotes the rotation of $w$ ending at $n-1$). Notice that $C_{w'}$ is a circular extension of $C_{\sdiv'}$.

    If the polygon $P$ is black, then by the inductive hypothesis, $w'$ has $k$ cyclic left descents. Because $P$ is black and $C_w$ is a circular extension of $C_\sdiv$, we must have $C_w=(1 \cdots n \cdots n-1 \cdots)$. So to obtain $w^{(n-1)}$ from $w'$, we put in $n$ somewhere to the right of $1$ and to the left of $n-1$. This means $n$ is a cyclic descent of $w^{(n-1)}$. So $w^{(n-1)}$ has $k+1$ cyclic descents, and so does $w$.

    The argument is very similar if $P$ is white. Then $w'$ has $k+1$ cyclic descents, and to obtain $w^{(n-1)}$ from $w'$, we put $n$ somewhere (cyclically) to the right of $n-1$ and to the left of $1$. This means $w^{(n-1)}$ and $w$ both have $k+1$ cyclic left descents.
\end{proof}

\cref{thm:when-w-simp-in-tile}, \cref{cor:when-w-cham-in-tile} and \cref{lem:circular-ext-sdiv-fixed-number-descents} imply the following. Recall from \cref{def:circular-extension} that we identify total cyclic orders $C_w$ with $n$-cycles $(w)$. 

\begin{corollary}\label{cor:w-simp-in-tile-cyclic-order-version}
Let $\sdiv$ be a bicolored subdivision of type $(k,n)$. Then
\[\htile{\sdiv}= \bigcup_{(w) \in \Ext(C_\sdiv)} \simp{w} \qquad \text{and} \qquad \atile{\sdiv}= \bigcup_{(w) \in \Ext(C_\sdiv)} \cham{w}.\]
That is, $\htile{\sdiv}$ (resp. $\atile{\sdiv}$) is the union of $w$-simplices $\simp{w}$ (resp. $w$-chambers $\cham{w}$) where the $w$-order $C_w$ is a circular extension of the $\sdiv$-order $C_{\sdiv}$.
\end{corollary}

\begin{proof}
    \cref{thm:when-w-simp-in-tile} and \cref{cor:when-w-cham-in-tile} imply that 
    \[\htile{\sdiv}= \bigcup\simp{w} \qquad \text{and} \qquad \atile{\sdiv}= \bigcup \cham{w}\]
    where the union is over 
    $\{w \in D_{k+1, n} \colon C_w \text{ a circular extension of }C_\sdiv\}.$
   \cref{lem:circular-ext-sdiv-fixed-number-descents} implies that this set is equal to 
   $\{v \in D_n \colon C_w \text{ a circular extension of }C_\sdiv\},$ which we identify with the set of ($n$-cycles which are) circular extensions of $C_\sdiv$.
\end{proof}

Because the $w$-simplices are unimodular, \cref{cor:w-simp-in-tile-cyclic-order-version} gives rise to an expression for the normalized volume of each hypersimplex tile.

\begin{corollary}\label{cor:vol-of-tile}
Let $\sdiv$ be a bicolored subdivision.
    The normalized volume of $\htile{\sdiv}$ is the number of circular extensions of $C_\sdiv$. That is, $\Vol(\htile{\sdiv})= |\Ext(C_\sdiv)|$.
\end{corollary}

We can also interpret \cref{prop:tilings} in terms of cyclic orders, which will be useful in \cref{sec:combPT}.
\begin{corollary}\label{cor:onek}
Let $w\in D_{k+1,n}$.
Then $C_w$ is a cyclic extension of exactly one of the  ``kermit'' cyclic orders
$C_{\sigma_I}$, where $I$ runs over ${[2,n-1] \choose k}$.  The same statement holds if we replace $\sigma_I$ by $\sigma_I^v$.
\end{corollary}
\begin{proof}
We know from \cref{prop:StanleyTriangulation} and \cref{prop:tilings} that 
$$\Delta_{k+1,n} = \bigcup_{w\in D_{k+1,n}} \simp{w} = \bigcup_{I \in {[2,n-1] \choose k}} \htile{\sigma_I}$$
where in both unions, the pieces have pairwise disjoint interiors.
Now the result follows from \cref{cor:w-simp-in-tile-cyclic-order-version}.
\end{proof}

Now if we use \cref{cor:onek} and let $k$ range over the interval $0\leq k \leq n-2$, we obtain the following.
\begin{corollary}\label{cor:allk}
Let $w \in D_n$, and choose $v\in [n]$.
Then $C_w$ is a circular extension of exactly one of the  ``kermit'' cyclic orders 
$C_{\sigma_I^v}$, where $I$ runs over all subsets of $[n]\setminus \{v\}$.
\end{corollary}

\begin{figure}[h]
\includegraphics[width=\textwidth]{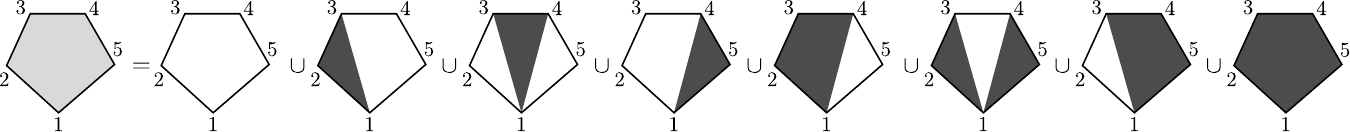}
        \caption{The all-grey tricolored subdivision of the $n$-gon corresponds to the 
        union of all kermit subdivisions for fixed $n$ and $0\leq k \leq n-2$.
        Here $v=1$.  We will see in \cref{ex:cube} and \cref{prop:tri-sub-polytope-as-union} that 
        this equation can be interpreted as a tiling of the unit hypercube $\mbox{\mancube}_{4} \subset \R^4$.}
        \label{fig:kermitsubdivisions}
\end{figure}

We represent \cref{cor:allk} by the diagram in \cref{fig:kermitsubdivisions}.
The left-hand side, which is the all-grey tricolored subdivision of the $n$-gon, represents the trivial cyclic order $C_{\emptyset}$  containing no triples.  Each of the eight bicolored subdivisions at the right is a kermit subdivision 
$\sigma_I$ and represents the partial cyclic order
$C_{\sigma_I}.$  \cref{cor:allk} then says that each circular extension of 
$C_{\emptyset}$ is a circular extension of precisely one of the partial cyclic orders $C_{\sigma_I}$ from the right-hand side.

\begin{remark}\label{rem:subcyclic}
There is a natural notion of \emph{sub-cyclic order} of a cyclic order.
If $C$ is a cyclic order on a set $X$, and $Y \subset X$
is a subset of $Y$, then $C$ restricts to a \emph{sub-cyclic order} $C|_Y$ on $Y^3$: namely,
$$C|_Y = \{(a,b,c) \ \vert \ a,b,c\in Y \text{ and }
(a,b,c)\in C\}.$$
Note that \cref{cor:allk} also applies to sub-cyclic orders of a cyclic order.
\end{remark}

\section{Magic Number Theorem for the $m=2$ amplituhedron}\label{sec:PT}

In this section, we show that every tiling of the hypersimplex $\Delta_{k+1, n}$ and every all-$Z$ tiling of $\mathcal{A}_{n,k,2}$ consists of $\binom{n-2}{k}$ tiles.

Recall the definition of the Parke-Taylor function from \cref{def:PT}.
Let $\mathbf{I}_n= 1 \dots n$ denote the identity permutation on $[n]$.
\begin{definition}\label{def:wtPT}
 Let $\simp{w}$ be a $w$-simplex for $\Delta_{k+1,n}$. 
 The \emph{weight function} of $\simp{w}$ is
 \begin{equation*}
     \Omega(\simp{w}):= \PT(w)
 \end{equation*}
 viewed as a rational function on $\widehat{\Gr}_{2,n}$, 
  the 
 affine cone over the Grassmannian $\Gr_{2,n}$.
\end{definition}

We note that the weight function $\Omega(\simp{w})$ is well-defined on $\unistr$,
the open locus in $\widehat{\Gr}_{2,n}$ where all Pl\"ucker coordinates are nonvanishing. We also remark that $\PT(w)$ depends only on the $n$-cycle $(w)$; that is, if $(u)=(w)$, then $\PT(w)=\PT(u)$.

\begin{remark}\label{rmk:PTgauge}
A point in $\unistr/T$ where $T=(\R^*)^n$
can be represented with a matrix $C$ whose top row is $(1,\ldots,1)$ and the bottom row is $(z_1,\ldots,z_n)$. Then the Pl\"ucker coordinates $P_{ij}$ of $C$ are simply $z_j-z_i$ and \cref{def:PT} reads:
\begin{equation}\label{eq:PT-in-x}
  \PT(w)=\frac{1}{(z_{w_2}-z_{w_1})(z_{w_3}-z_{w_2})\ldots (z_{w_{n}}-z_{w_{n-1}})(z_{w_1}-z_{w_n})}.
\end{equation}
Some of our proofs about Parke-Taylor functions will use the formulation of \eqref{eq:PT-in-x}. 
\end{remark}

\begin{definition}\label{def:weight_pos_pol}
    Let $\Gamma_S \subset \Delta_{k+1, n}$ be a full-dimensional positroid polytope,
    or any other subset of $\Delta_{k+1,n}$ which is a union of $w$-simplices  (cf. 
    \cref{prop:w-simp-triangulate-pos-poly}).  We define the \emph{weight function} of $\Gamma_S$ to be the sum of the weight functions of $w$-simplices included in $\Gamma_S$:
    \begin{equation*}
\Omega(\Gamma_S):=\sum_{\simp{w} \subset \Gamma_S} \Omega(\simp{w})
    \end{equation*}
\end{definition}

The first step in our argument is to show that all tiles have the same weight function. We contrast this with the normalized volume of tiles, which is far from constant (see \cref{fig:tiling}).

\begin{proposition}\label{prop:weight_tile}
  Let $\htile{\sdiv}$ be a tile of $\Delta_{k+1,n}$. Then the weight function of $\htile{\sdiv}$ is
  \begin{equation}\label{eq:weight_tile}
      \Omega(\htile{\sdiv})=(-1)^k \PT(\mathbf{I}_n).
  \end{equation}
  In particular, all tiles of $\Delta_{k+1,n}$ have the same weight function.
\end{proposition}

We need some preliminaries before proceeding to the proof of \cref{prop:weight_tile}.
\begin{definition}
    For a function $F\in\mathbb{C}(\widehat{\Gr}_{2,n})$  
    all of whose poles are simple, we define its \emph{residue} at $P_{ij}=0$ by
    \[\Res_{P_{ij}=0}F:=\lim_{P_{ij\to 0}}P_{ij}F.\]
\end{definition}
This definition resembles, of course, the usual definition of a residue, only that we define it for functions rather than forms.

 \begin{definition}\label{def:Ext}
 Let $C$ be a cyclic order on $[n]$, and let $i,j\in [n].$
We let  $$\Ext_{(ij)}(C):=\{C' \in \Ext(C): C'= (\cdots i\ j \cdots)\}$$ 
 be the set of circular extensions of $C$ in which $i$ immediately precedes $j$. 
 \end{definition}

The next lemma constrains the poles of $\Omega(\htile{\sdiv})$.
\begin{lemma}\label{lem:no_poles_except_in_bdries_of_polygons} Let $\sdiv$ be a bicolored subdivision of $\mathbf{P}_n$. Then
\[\Res_{P_{ij}=0}\Omega(\htile{\sdiv})=0,\] for all $i,j$ such that $(i,j)$ is not an edge of a white or black polygon of $\sigma$. 
\end{lemma}

\begin{proof}
First, suppose 
$(i,j)$ is a diagonal (but not an edge) of a (black or white) polygon in $\sigma$. 
Consider $\simp{w} \subset \htile{\sdiv}$. By \cref{thm:when-w-simp-in-tile}, $i,j$ are not adjacent in $C_w$, hence $P_{ij}=0$ is not a pole of $\Omega(\simp{w})$. It follows that also $\Omega(\htile{\sdiv})$ does not have a pole at $P_{ij}=0$.

Now suppose $i,j$ are not vertices of the same polygon in $\sigma$. 
By \cref{cor:w-simp-in-tile-cyclic-order-version}, the $w$-simplices $\simp{w} \subset \Gamma_\sigma$ such that $\PT(w)$ has  poles at $P_{ij}=0$ are exactly those for which $(w) \in \Ext_{(ij)}(C_\sigma) \cup \Ext_{(ji)}(C_\sigma)$. We define the map 
\begin{align*}
    \phi \colon \Ext_{(ij)}(C_\sigma) &\to \Ext_{(ji)}(C_\sigma)\\
    C_w=(\cdots i\ j \cdots) & \mapsto (\cdots j\ i \cdots)=:C_{\phi(w)}
\end{align*}
which swaps the order of $i$ and $j$.
The map $\phi$ is well-defined and a bijection: $w$ satisfies \cref{thm:when-w-simp-in-tile} if and only if $\phi(w)$ does, since we are not changing the order of numbers which are vertices of the same polygon. 
(The geometric interpretation of $\phi$---see \cref{cor:w-simp-in-tile-cyclic-order-version} and \cref{prop:w-simp-facet-sharing}---is that $\simp{w}$ and $\simp{\phi(w)}$ are $w$-simplices in $\htile{\sdiv}$ which intersect in a common facet on a hyperplane $x_{[i, j-1]}=d$.) We claim that for all $(w) \in \Ext_{(ij)}(C_\sigma)$
\begin{equation}\label{eq:pole_cancellation}\Res_{P_{ij}=0}\left(\Omega(\simp{w})+ \Omega(\simp{\phi(w)})\right)=0.\end{equation}
It then will follow immediately from
\eqref{eq:pole_cancellation} that $\Omega(\htile{\sdiv})$ does not have a pole at $P_{ij}=0$. 

To show \eqref{eq:pole_cancellation}, suppose $C_w= (\cdots h\ i\ j \ l\  \cdots)$ and $C_{\phi(w)}= (\cdots h\ j\ i\ l\  \cdots)$. Let $Q$ be the product of common factors in $\Omega(\simp{w})$ and $\Omega(\simp{\phi(w)})$, which are of the type $1/P_{rs}$, with $i,j \not \in \{r,s\}$. Then
\begin{align*}\Omega(\simp{w})+\Omega(\simp{\phi(w)})&=Q\left(\frac{1}{P_{h,i}P_{i,j}P_{j,l}}+\frac{1}{P_{h,j}P_{j,i}P_{i,l}}\right)=\frac{Q}{P_{i,j}}\frac{P_{h,j}P_{i,l}-P_{h,i}P_{j,l}}{P_{h,i}P_{j,l}P_{h,j}P_{i,l}}=\\
&=\frac{Q}{P_{i,j}}\frac{P_{h,l}P_{i,j}}{P_{h,i}P_{j,l}P_{h,j}P_{i,l}}
=\frac{QP_{h,l}}{P_{h,i}P_{j,l}P_{h,j}P_{i,l}},\end{align*}
where in the third equality we used the Pl\"ucker relation $P_{h,j}P_{i,l}-P_{h,i}P_{j,l}+P_{h,l}P_{i,j}=0$. Therefore $\Omega(\simp{w})+\Omega(\simp{\phi(w)})$ does not have a pole at $P_{ij}=0$, which shows \eqref{eq:pole_cancellation}.
\end{proof}

The next proof will use the formulation from \eqref{eq:PT-in-x}, so we introduce the following notation.
  \begin{notation}\label{notation:pt}
  We define 
  the notation $\omega(\simp{w}):= \Omega(\simp{w})|_{P_{ij} \mapsto (z_j-z_i)}$,
  $\omega(\htile{\sdiv}):= \Omega(\htile{\sdiv})|_{P_{ij} \mapsto (z_j-z_i)}$, and $\pt(w):= \PT(w)|_{P_{ij} \mapsto (z_j-z_i)}.$
\end{notation}

\begin{proof}[Proof of \cref{prop:weight_tile}] 
Consider the function 
\[F=\Omega(\htile{\sdiv})-(-1)^k\PT(\mathbf{I}_n),\]
which is a rational function on $\widehat{\Gr}_{2,n}.$ We would like to show that $F$ is the zero function. We first perform a series of reductions. 

 First, it suffices to show that $F(V)=0$ for all $V \in \unistr$, the dense subset of $\widehat{\Gr}_{2,n}$ where all Pl\"ucker coordinates are nonvanishing. Note that $F$ is well-defined on this subset. 

Choose $V\in \unistr$ and let $A$ be a matrix representative. Recall that the torus $(\mathbb{R}^*)^n$ acts on $\widehat{\Gr}_{2,n}$ by rescaling columns. Scaling a column of $A$ by $t \in \mathbb{R}^*$ will multiply $F(A)$ by $t^{-2}$. This gives the second reduction: to show $F(V)$ is zero, it suffices to show that $F$ is zero on a single point in the torus orbit of $V$.

The torus orbit of $V$ contains a point of the form 
\[B=\begin{bmatrix}
    1 & 1 & \cdots & 1 & 1\\
    z_1 & z_2 & \cdots &z_{n-1} & z_n\\
\end{bmatrix}\]
where $z_1, \dots, z_n \in \mathbb{R}$.
Note that the Pl\"ucker coordinate $P_{ij}(B)$ is equal to $(z_j-z_i)$. This is the third reduction: to show that $F(B)$ is zero, it suffices to show that $$f(z_1, \dots, z_n):=F|_{P_{ij} \mapsto (z_j - z_i)}$$ is zero as a function on $\mathbb{R}^n$.

To show that $f(z_1, \dots, z_n)=0$, we proceed by induction on $n.$ The base case $n=3$ is simple and can be checked by hand. Now suppose we have proven $f=0$ for all $n'<n.$ In what follows, we will deduce $f=0$ for $n$ by carefully analyzing the poles of $f$.

Because $\PT$ functions are homogeneous of degree $-n$ in Pl\"ucker coordinates, $f$ is either identically zero or of degree $-n$ in the $z_i$. Because $\PT$ functions have (only) simple poles of the form $P_{ij}=0$, any poles of $f$ are simple and lie at $z_j-z_i=0$ for some $i,j$.
\cref{lem:no_poles_except_in_bdries_of_polygons} implies that all poles of $\Omega(\htile{\sdiv})$ are of the form $P_{ij}=0$ where $(i,j)$ is an edge of some polygon in $\sdiv$. The poles of $\PT(\mathbf{I}_n)$ are at $P_{i-1,i}=0$. So the only possible poles of $f$ are $z_j-z_i=0$ for $(i,j)$ an edge of a polygon in $\sdiv$ (this includes the case $z_i-z_{i-1}=0$).

We claim that $f$ does not have poles of the form $z_{i}-z_{i-1}=0$. That is, we have the following.

\begin{claim}\label{lem:residue_in_ii+1}
    Let $i \in [n]$. Then, using \cref{notation:pt},
   we have 
    $$\lim_{z_i\to z_{i-1} }(z_i - z_{i-1}) \ \omega(\htile{\sdiv})= \lim_{z_i\to  z_{i-1}} (z_i - z_{i-1}) (-1)^{k}\pt(\mathbf{I}_n)$$
    and thus $\lim_{z_i \to  z_{i-1}} (z_i -z_{i-1})f=0.$
\end{claim}

Before proving \cref{lem:residue_in_ii+1}, we explain why this claim implies that $f$ is zero. Indeed, we can write $f=\frac{P}{Q},$ where $P$ is a homogenous polynomial in $z_1,\ldots, z_n$ of degree $h\geq0,$ and $Q$ is a product of factors of the form $(z_j-z_i),$ where each factor appears at most once. Suppose for the sake of contradiction that $f$ is not zero, and thus is of degree $-n$. Then the degree of $Q$ is $h+n\geq n.$ By \cref{lem:residue_in_ii+1}, the factors of $Q$ are $(z_j-z_i)$ where $(i,j)$ is an edge of a polygon in $\sdiv$ but is not an edge of $\mathbf{P}_n$. There are at most $n-3$ such arcs, so $\text{deg}(Q)\leq n-3$. But this contradicts the fact that $\text{deg}(Q)\geq n.$

We now prove \cref{lem:residue_in_ii+1}, which will conclude the proof.  We would like to compute the sum over $\simp{w} \subset \htile{\sdiv}$ of terms
\[
\lim_{z_i \to z_{i-1}}(z_i-z_{i-1})\omega(\simp{w}).\]

Fix $\simp{w} \subset \htile{\sdiv}$. If $i-1,i$ are not adjacent in the cyclic order $C_w$, then 
\[\lim_{z_i \to z_{i-1}}(z_i-z_{i-1})\omega(\simp{w})=0.\]

If $i-1,i$ are adjacent in $C_w$, their order is dictated by the color of the polygon $p$ in $\sdiv$ containing the edge $(i-1,i)$. If $p$ is white, then 
$C_w=(\cdots i-1\ i\cdots) \in \Ext_{(i-1,i)}(C_\sdiv)$;
if $p$ is black, then
$C_w=(\cdots i\ i-1 \cdots)\in \Ext_{(i,i-1)}(C_\sdiv)$.
Let $c=0$ if $p$ is white and let $c=1$ if $p$ is black.

Consider now the limit $z_i\to z_{i-1}.$ 
Write $w=w_1\dots w_n$ with the $i$ and $i-1$ in positions $\{j,j+1\}.$
It is immediate that 
\begin{small}
\begin{align}\label{eq:lim1}
\begin{split}
     \lim_{z_i\to z_{i-1}}(z_i-z_{i-1})\omega(\simp{w}) &= \lim_{z_{w_j}\to z_{w_{j+1}}} \frac{(-1)^c}{(z_{w_{2}}-z_{w_{1}})(z_{w_3}-z_{w_2})\ldots (z_{w_{j}}-z_{w_{j-1}})(z_{w_{j+2}}-z_{w_{j+1}})\ldots (z_{w_1}-z_{w_n})}\\
     &= \frac{(-1)^c}{(z_{w_{2}}-z_{w_{1}})(z_{w_3}-z_{w_2})\ldots (z_{w_{j+1}}-z_{w_{j-1}})(z_{w_{j+2}}-z_{w_{j+1}})\ldots (z_{w_1}-z_{w_n})}.
\end{split}
\end{align}
\end{small}

Let $\sdiv'$
be the bicolored subdivision of the $(n-1)$-gon with vertices $N'=\{1, 2, \dots, i-2, \star, i+1, \dots, n\}$ obtained from $\sdiv$ by contracting the boundary edge $(i-1, i)$ to the vertex $\star.$
Let $\rho$ be the map
\begin{equation*}
   \rho: \{C_w \in C_\sdiv:  i-1,i \text{ adjacent in }C_w\} \to \Ext(C_{\sdiv'})
   \end{equation*}
where $\rho(C_w):=C_{w'}$ is the cyclic order on $N'$ obtained from $C_w$ by identifying $i-1$ and $i$ with $\star$. It is easy to see that $\rho$ is a bijection. Then \eqref{eq:lim1} can be rewritten as
\begin{equation}\label{eq:contracting_pt}\lim_{z_i \to z_{i-1}}(z_i-z_{i-1})\omega(\simp{w})=(-1)^{c}\omega(\simp{w'}).\end{equation}
   
By the induction hypothesis
\begin{equation}\label{eq:ind_hp}
 \omega(\Gamma_{\sdiv'})=\sum_{(u)\in \Ext(C_{\sdiv'})}\omega(\simp{u}) = (-1)^{k-c}\pt(\mathbf{I}_{N'}).   
\end{equation}

Putting everything together, we have 
\begin{align*}
    \lim_{z_i \to z_{i-1}}(z_i - z_{i-1})\omega(\htile{\sdiv}) &= \sum_{\substack{(w) \in \Ext(C_\sdiv):\\i-1, i \text{ adj in } C_w}}
    \lim_{z_i \to z_{i-1}}(z_i - z_{i-1})\omega(\simp{w}) \\ &=\sum_{\substack{(w) \in \Ext(C_\sdiv):\\i-1, i \text{ adj in } C_w}} (-1)^c \omega(\simp{w'})\\
    &= (-1)^c\sum_{(u) \in \Ext(C_{\sdiv'})}\omega(\simp{u})\\
    &= (-1)^k \pt(\mathbf{I}_{N'}).
\end{align*}
In the second equality we use \eqref{eq:contracting_pt}, in the third the fact that $\rho$ is a bijection, and in the fourth we use \eqref{eq:ind_hp}. By inspection, the last line is equal to $$\lim_{z_i \to z_{i-1}} (z_i - z_{i-1}) (-1)^k \pt(\mathbf{I}_n),$$ which proves \cref{lem:residue_in_ii+1}.  
\end{proof}

\begin{proposition}\label{prop:samesize}
Let $R \subset \Delta_{k+1,n}$ be a subset of $\Delta_{k+1,n}$ which admits a positroid tiling, i.e. it can be written 
as the union
of positroid tiles $\{\htile{\sdiv}\}_{\sdiv \in \mathcal{S}}$ whose interiors are disjoint.  Then all positroid tilings of $R$ have the same cardinality.    
\end{proposition}
\begin{proof}
For \emph{any} tiling $\{\htile{\sdiv}\}_{\sdiv \in \mathcal{S}}$ of $R$, we have
    \begin{equation*}
        \Omega(R)=\sum_{\sdiv \in \mathcal{S}} \sum_{\simp{w} \subset \htile{\sdiv}}\Omega(\simp{w})=\sum_{\sdiv \in \mathcal{S}} \Omega(\htile{\sdiv})=|\mathcal{S}| (-1)^k \PT(\mathbf{I}_n),
    \end{equation*}
    where for the first and second equality we used that $\htile{\sdiv}=\cup_{\simp{w} \subset \htile{\sdiv}} \simp{w}$ and that the tiles $\{\htile{\sdiv}\}_{\sdiv \in \mathcal{S}}$ have disjoint interiors and cover $\Delta_{k+1,n}$. For the last equality we used \cref{prop:weight_tile}.
    
    It follows that each positroid tiling of $R$ must have the same cardinality.
    \end{proof}

\begin{theorem}\label{thm:magic-number}
Every positroid tiling of $\Delta_{k+1, n}$ consists of $\binom{n-2}{k}$ tiles.
\end{theorem}
\begin{proof}
By \cref{prop:samesize}, each tiling of $\Delta_{k+1,n}$ has the same cardinality. 
    Since the kermit tiling in \cref{prop:tilings} has cardinality ${n-2 \choose k}$, we conclude that all tilings of $\Delta_{k+1,n}$ have this cardinality. 
\end{proof}

\begin{example}
Consider the collection $\mathcal{S}$ of bicolored subdivisions in \cref{fig:tiling}. Then $\{\htile{\sdiv}\}_{\sdiv \in \mathcal{S}}$ is a tiling for $\Delta_{4,7}$. This can be verified by checking that each $w$-simplex in $\Delta_{4,7}$ lies in a unique tile $\htile{\sdiv}$ with $\sdiv \in \mathcal{S}$. By T-duality (cf. \cref{thm:Tduality}), $\{\atile{\sdiv}\}_{\sdiv \in \mathcal{S}}$ is an all-Z tiling of $\mathcal{A}_{7,3,2}$. The tiling contains $10={7-2 \choose 3}=M_{7,3,2}$ tiles (cf. \cref{cor:magic-number}).
\begin{figure}[h]
\includegraphics[width=0.9\textwidth]{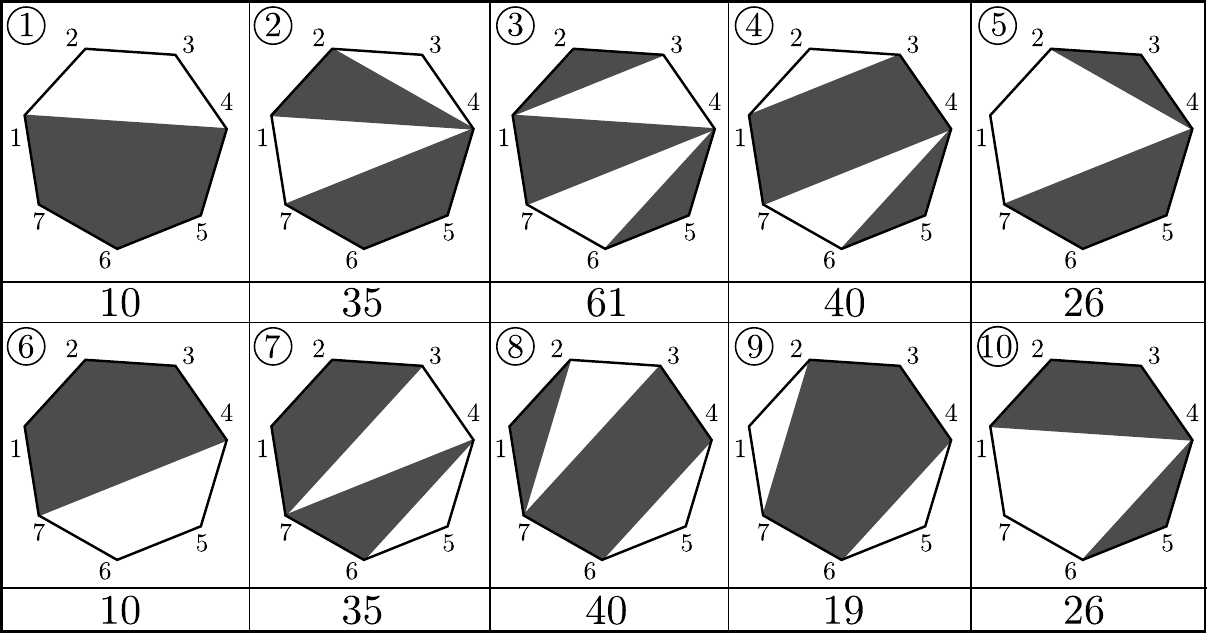}
        \caption{A collection of $10={7-2 \choose 3}=M_{7,3,2}$ bicolored subdivisions of type $(3,7)$ (labelled from $\textcircled{\raisebox{-0.9pt}{1}}$ to $\textcircled{\raisebox{-0.9pt}{10}}$) which gives a tiling for $\Delta_{4,7}$ and $\mathcal{A}_{7,3,2}$. The number in the box below each bicolored subdivision is the volume of the corresponding positroid polytope in $\Delta_{4,7}$. The sum of their volumes is the Eulerian number $E_{3,6}=302$.}
        \label{fig:tiling}
\end{figure}
\end{example}

 Using results on the positive Dressian, one can show that every full-dimensional positroid polytope has a positroid tiling. This will allow us to apply \cref{prop:samesize} to full-dimensional positroid polytopes as well.

 The proof of the following lemma assumes some knowledge
 of plabic graphs; see \cite{FWZ7} for background. A matroid is \emph{series-parallel} matroid if it does not have the uniform matroid $U_{2,4}$ or the graphical matroid $M_{K_4}$ as a minor.
\begin{lemma}\label{lem:tree}
Suppose that a positroid $\mathcal{M}$ is a connected matroid.  
If $\mathcal{M}$  
is series-parallel,
then it comes from a reduced plabic graph $G$ which is a tree. 
\end{lemma}
\begin{proof}
Since $\mathcal{M}$ is connected, it comes from a reduced plabic graph $G$
which is also be connected \cite{ARW}.  Assume that $G$ is not a tree.
Then it must have an internal cycle, in which case we can delete some 
edges to obtain a subgraph $G'$ which, apart from lollipops on 
the boundary, is the plabic graph of the top cell of $\Gr_{2,4}^{\geq 0}$
(that is, it has a bipartite $4$-cycle with each of the four vertices
of the cycle attached by a single path to a boundary vertex).
The cell $S_{G'}$ is contained in the boundary of $S_G$,
and hence the positroid polytope $\Gamma_{G'}$ (an octahedron)
either has the same dimension
as $\Gamma_G$, or it is 
contained in the boundary of the positroid polytope $\Gamma_{G}$.

If $\dim \Gamma_{G'} = \dim \Gamma_G$, then the positroids of $G$ and of 
$G'$ have the same
number of connected components, which implies that $G'$ is connected
(no lollipops on the boundary) and hence  $G=G'$. This implies that $\mathcal{M}$ is the uniform matroid $U_{2,4}$, which is not series-parallel.
On the other hand, if the octahedron $\Gamma_{G'}$ is contained 
in the boundary of the positroid polytope $\Gamma_{G}$, then 
$\Gamma_G$ has a face which is an  octahedron, hence
by \cite[Corollary 6.5]{SW21}, $\mathcal{M}$ is not series-parallel.
\end{proof}
The converse of \cref{lem:tree} is also true, but we do not need it here.

\begin{proposition}\label{lem:contains}
Let $\Gamma_{\mathcal{M}} \subset \Delta_{k+1,n}$ be a 
positroid polytope.
Then there is a positroid subdivision of $\Delta_{k+1,n}$
(that is, a polyhedral subdivision of $\Delta_{k+1,n}$ into 
positroid polytopes) that contains $\Gamma_{\mathcal{M}}$.
Moreover, each full-dimensional positroid polytope admits a 
positroid tiling.
\end{proposition}
\begin{proof}
The proof of \cite[Theorem 5.1]{SW21} uses the rank
function of the matroid
$\mathcal{M}$ to construct a point $P$ in the positive Dressian,
such that the corresponding regular matroid subdivision $\mathcal{D}_P$
of $\Delta_{k+1,n}$ contains $\Gamma_{\mathcal{M}}$ as a face.  By \cite[Theorem 9.11]{LPW}, each
face of $\mathcal{D}_P$ is a positroid polytope and so $\mathcal{D}_P$ is a regular positroid subdivision.

To prove the second statement, we choose a finest regular positroid subdivision $\mathcal{D}$ of $\Delta_{k+1,n}$ (i.e. $\mathcal{D}$ has no non-trivial refinement which is a regular positroid subdivision) which refines $\mathcal{D}_P$. Such a finest subdivision exists as there are finitely many regular positroid subdivisions of $\Delta_{k+1,n}$.  We claim that each full-dimensional positroid 
polytope $\Gamma_{\mathcal{M'}}$ 
in $\mathcal{D}$ is a tile.
To see this, 
note that by \cite[Theorem 6.6]{SW21}, 
$\Gamma_{\mathcal{M'}}$
must be a connected {series-parallel} matroid. \cref{lem:tree} then implies that 
$\mathcal{M'}$
must come from a plabic graph which is a tree. 
Therefore by \cite[Proposition 3.16]{LPW}, 
$\Gamma_{\mathcal{M'}}$ is a tile. Now, restricting $\mathcal{D}$ to $\Gamma_{\mathcal{M}}$, we obtain a tiling of $\Gamma_{\mathcal{M}}$.

\end{proof}

\begin{corollary}\label{cor:positroidpolytopes}
Let $\Gamma_{\mathcal{M}} \subset \Delta_{k+1,n}$ be a 
full-dimensional positroid polytope.
Then $\Gamma_{\mathcal{M}}$ admits a positroid tiling,
and all positroid tilings 
of $\Gamma_{\mathcal{M}}$ 
have the same cardinality.
\end{corollary}

\begin{proof}
This follows from \cref{lem:contains}
and \cref{prop:samesize}.
\end{proof}
It would be interesting to address the following question.
\begin{question}
Is there an explicit combinatorial formula for the cardinality 
of a tiling of a full-dimensional positroid polytope 
 $\Gamma_{\mathcal{M}} \subset \Delta_{k+1,n}$? 
\end{question}

We can give analogues of some of the previous results for the amplituhedron. 

\begin{proposition}\label{prop:ampanalogue}
If $R$ is a full-dimensional subset of $\mathcal{A}_{n,k,2}(Z)$ which admits an all-$Z$ positroid tiling, then every all-$Z$ positroid tiling of $R$ has the same cardinality.  
\end{proposition}

\begin{proof}[Proof sketch]
    In analogy to \cref{def:wtPT}, define the weight function of a $w$-chamber $\cham{w}$ to be $\Omega(\cham{w})= \PT(w).$
Since $R \subset \Ank$ admits an all-$Z$ positroid tiling and each tile is a union of $w$-chambers, for any $Z$ $R$ is a union of $w$-chambers. Define the weight function of $R$ to be
\[\Omega(R)= \sum \Omega(\cham{w}) \]
where the sum is over chambers $\cham{w}$ such that for some $Z$, $\cham{w} \cap R$ is nonempty. Then, using \cref{rmk:wsimp-in-tile-iff-wcham-in-tile}, we may deduce an amplituhedron analogue of \cref{prop:weight_tile} on the weight of each tile. Finally, by an identical argument as in the hypersimplex case, we obtain the proposition statement.
\end{proof}

\cref{prop:ampanalogue} together with \cref{prop:tilings} allows us to deduce the Magic Number Conjecture for the $m=2$ amplituhedron.  
Alternatively, we can prove the Magic Number Conjecture using \cref{thm:Tduality} and \cref{thm:magic-number}.  
\begin{corollary}\label{cor:magic-number}
 Every all-$Z$ positroid tiling of $\mathcal{A}_{n,k,2}$ consists of $M_{n,k,2}= \binom{n-2}{k}$ tiles.   
\end{corollary}

\begin{remark}
    Generalizing the notion of a tile for the amplituhedron,
we define a \emph{Grasstope} $Z_S:=\overline{\tilde{Z}(S)}$ to be (the closure of) the 
image of a positroid cell $S$ under the amplituhedron map $\tilde{Z}$. 
We expect that every 
full-dimensional Grasstope in 
$\mathcal{A}_{n,k,2}(Z)$ admits a tiling, but this does not follow from our results.
\end{remark}

\begin{corollary} \label{cor:Grasstopes}
	Let $Z_S$ be a $2k$-dimensional (full-dimensional) Grasstope in 
$\Ank$.
Then every all-$Z$ tiling of 
$Z_S$ has the same cardinality.  
\end{corollary}
\begin{proof}
Either $Z_S$ has no all-$Z$ tilings, in which case the statement is vacuously true, or this follows from \cref{prop:ampanalogue}.
\end{proof}

\section{Combinatorial results on tilings}\label{sec:combtiling}

In the last section, we showed that any tiling of $\Delta_{k+1, n}$ or $\A_{n, k, 2}$ consists of $\binom{n-2}{k}$ tiles. In this section, we prove another enumerative necessary condition for a collection of tiles to comprise  a tiling. The strategy is again to use Parke-Taylor functions, but now for facets of tiles. 

We start by defining a weight function for common facets of $w$-simplices. Recall 
the description of these facets 
from \cref{prop:w-simp-facet-sharing}.

\begin{notation}\label{not:facetweight}
Choose $w\in S_n$ and let $i$ and $j$ be letters located in adjacent positions in $w$.
Let $N_L:=[i+1,j-1]\cup\{\star\}$ and $N_R:=[j+1,i-1]\cup\{\star\}$ be subsets of $[n]/\sim$, where 
$\sim$ is the equivalence relation on $[n]$ obtained by identifying $i$ and $j$, and $\star:=i \sim j$.
 We define a total cyclic order 
 $C_L$ on the set $N_L$ by taking the subword of $w$ consisting of letters in $[i,j]$, and then identifying 
 $i \sim j \sim \star$.  Similarly, we define a total cyclic order $C_R$ on $N_R$ by taking the subword of $w$ consisting of letters in $[j,i]$,
 and then identifying $i \sim j \sim \star$.
We also define $(w_L)$ to be the $|N_L|$-cycle on $N_L$ such that $C_{L}=(w_L)$. We define $(w_R)$ analogously.
\end{notation}

\begin{example}\label{ex0:weight_facet_w}
Let $n=6$, $i=1$, and $j=4$.  
Let $w=514236$. 
We have $N_L=\{2,3,\star\}, N_R=\{5,6,\star\}$ and  $C_L=(\star \  2\ 3)=(w_L)$ and $ C_R=(5 \  \star \ 6)=(w_R)$.
\end{example}

In the next definition, we use the notation 
$\widehat{\Gr}_{2,N_L}$ and $\widehat{\Gr}_{2,N_R}$ to denote Grassmannians whose matrix representatives have columns labeled by $N_L$, respectively $N_R$, rather than $[n]$.

\begin{definition}\label{def:facetweight}
Let $\simp{w}^{(ij)}= \simp{u}^{(ji)}$ be a common facet of $\simp{w}$ and $\simp{u}$. Then we define the \emph{weight function} of $\simp{w}^{(ij)}$ to be the rational function on $\widehat{\Gr}_{2,N_L} \times \widehat{\Gr}_{2,N_R}$ given by
\begin{equation*} \Omega(\simp{w}^{(ij)}):=\PT(w_L)\PT(w_R) = \PT(u_L)\PT(u_R),
\end{equation*}
where we used \cref{not:facetweight}. 
\end{definition}
Because $\PT(u)$ depends only on the cycle $(u)$, $\PT(w_L)$ and $\PT(w_R)$ are well-defined.

\begin{example}\label{ex:weight_facet_w}
We continue \cref{ex0:weight_facet_w}.
Consider the facet $\Delta_{(w)}^{(14)}$ common to $\simp{w}$ and $\simp{u}$, where $w=514236$ and $u=541236$. We have $\simp{w} \subset \Gamma_\sigma$ where $\sdiv$ is as in \cref{ex:order}. Since $(w_L)=(\star \  2\ 3)$ and $ (w_R)=(5 \  \star \ 6)$,
\begin{equation*} \Omega(\simp{w}^{(14)})=\PT(w_L)\PT(w_R)=\frac{1}{P_{\star 2} P_{23} P_{3\star}}\frac{1}{P_{5\star} P_{\star 6} P_{65}}.   
\end{equation*}  
\end{example}

Recall from \cref{thm:hypersimplex-tiles-ineq-facets} that the tile $\htile{\sdiv}$ has a facet on the hyperplane $x_{[i,j-1]}=\area(i \rightarrow j)$ for each facet-defining arc $i \to j$. Denote this facet by $\htile{\sdiv}^{(ij)}$. Before defining the weight function for $\htile{\sdiv}^{(ij)}$, we characterize which $w$-simplices have a facet contained in $\htile{\sdiv}^{(ij)}$. Recall the notation $\Ext_{(ij)}(C_\sdiv)$ from \cref{def:Ext}.

\begin{lemma}\label{lem:omegafacet}
    Let $\sdiv$ be a bicolored subdivision and $i \to j$ an internal facet-defining arc. Then $\simp{w} \subset \htile{\sdiv}$ has a facet contained in $\htile{\sdiv}^{(ij)}$ if and only if $C_w \in \Ext_{(ij)}(C_\sdiv)$.
\end{lemma}
\begin{proof}

    Suppose $\simp{w} \subset \htile{\sdiv}$ has a facet on the hyperplane $x_{[i,j-1]}=\area(i \to j)$. This implies by \cref{cor:ineq-and-facet-w-simp} that $i,j$ are adjacent in $C_w$, so we have either $C_w= (\cdots i\ j \cdots)$ or $C_w=(\cdots j\ i \cdots)$. In the former case, \cref{cor:ineq-and-facet-w-simp} implies $\simp{w}$ has a facet inequality $x_{[i,j-1]} \geq \area(i \to j)$; in the latter, $\simp{w}$ has a facet inequality of the form $ x_{[i, j-1]} \leq \area(i \to j)$. Since $\simp{w}$ satisfies the facet inequalities of $\htile{\sdiv}$, we have $x_{[i,j-1]} \geq \area(i \to j)$ and $C_w= (\cdots i\ j \cdots)$ as desired. 

    Suppose $\simp{w} \subset \htile{\sdiv}$ and $C_w= (\cdots i\ j \cdots)$. Then $\simp{w}$ satisfies the inequalities 
    \[\area(i \to j) \leq x_{[i, j-1]} \leq \area(i \to j)+1\]
    which hold for all points in $\htile{\sdiv}$, and also has a facet on a hyperplane $x_{[i,j-1]}=c$ by \cref{cor:ineq-and-facet-w-simp}. Since $c$ is an integer, it is either equal to $\area(i \to j)$ or $\area(i \to j) +1.$ The hyperplane $x_{[i, j-1]} = \area(i \to j)+1$ intersects the tile $\htile{\sdiv}$ in codimension at least 2, as it is a bounding hyperplane but not a facet hyperplane. Therefore, $c= \area(i \to j)$ and $\simp{w}$ has a facet on the desired hyperplane.
\end{proof}

\begin{definition}
Let $\sdiv$ be a bicolored subdivision and $i \to j$ an internal facet-defining arc. Let $\htile{\sdiv}^{(ij)}$ denote the facet of $\htile{\sdiv}$ on the hyperplane $x_{[i,j-1]}=\area(i \rightarrow j)$. Then we define the weight function of the facet $\htile{\sdiv}^{(ij)}$ as the rational function on $\widehat{\Gr}_{2,N_L} \times \widehat{\Gr}_{2,N_R}$ given by the sum 
\begin{equation*}
\Omega(\htile{\sdiv}^{(ij)}):=\sum_{\simp{w}^{(ij)} \subset \htile{\sdiv}^{(ij)}} \Omega(\simp{w}^{(ij)})=\sum_{(w) \in \Ext_{(ij)}(C_\sdiv)} \Omega(\simp{w}^{(ij)})
\end{equation*}
where the equality follows from \cref{lem:omegafacet}.
\end{definition}

We now show the facet analogue of \cref{prop:weight_tile}.

\begin{notation}
    Using \cref{not:facetweight}, we define $\shuff:= \binom{n-2}{|N_L|-1}= \binom{n-2}{|N_R|-1}.$ We also denote by $\mathbf{I}_{N_L},\mathbf{I}_{N_R}$ the identity permutations on $N_L,N_R$ respectively.
\end{notation}

\begin{proposition}\label{prop:weight_facet}
Let $\sdiv$ be a bicolored subdivision and $i \rightarrow j$ an internal facet-defining arc. Let $\htile{\sdiv}$ be the corresponding tile for $\Delta_{k+1,n}$. Then
\begin{equation*}
    \Omega(\htile{\sdiv}^{(ij)})=(-1)^{k-1}\shuff\PT(\mathbf{I}_{N_L})\PT(\mathbf{I}_{N_R}).
\end{equation*}
In particular, the weight function $\Omega(\htile{\sdiv}^{(ij)})$ does not depend on $\sdiv$.
\end{proposition}
\begin{proof} 

By the definition of $\Omega(\htile{\sdiv}^{(ij)})$ and $\Omega(\simp{w}^{(ij)})$, we have 
\[\Omega(\htile{\sdiv}^{(ij)}) = \sum_{(w) \in \Ext_{(ij)}(C_\sdiv)} \PT(w_L) \PT(w_R).\]
As the first step in our argument, we would like to rewrite this as a product of weight functions for tiles with smaller $k$ and $n$.

Let $\sdiv_L$ (resp. $\sdiv_R$) be the bicolored subdivision on a polygon with vertices $N_L$ (resp. $N_R$) obtained from $\sdiv$ by taking all polygons to the left (resp. to the right) of the arc $i \rightarrow j$ and then contracting the edge $(i,j)$ to a vertex $\star$, see \cref{fig:fact_gen}. We write $k_L$ (resp. $k_R$) for the area of $\sdiv_L$ (resp. $\sdiv_R$).

We will show that 
\begin{equation} \label{eq:facet-weight-fcn-factors}
    \Omega(\htile{\sdiv}^{(ij)})= \shuff \, \Omega(\Gamma_{\sdiv_L}) \,  \Omega(\Gamma_{\sdiv_R}).
\end{equation}

The terms of the left hand side are labeled by $(w) \in \Ext_{(ij)}(C_\sdiv)$, while terms on the right hand side are labeled by elements of $\Ext(C_{\sdiv_L})\times \Ext(C_{\sdiv_R})$.
 We define a map 
 $$\psi :\Ext_{(ij)}(C_\sdiv) \to \Ext(C_{\sdiv_L})\times \Ext(C_{\sdiv_R})$$ which sends $C_{w}$ to $(C_{w_L},C_{w_R})$. 
 This map is well-defined and surjective by \cref{thm:when-w-simp-in-tile}. Each pair $(C_{u}, C_{v})$ in the image has exactly $\shuff=\binom{n-2}{|N_L|-1}$ pre-images. Indeed, if $C_{u}=(u_1 \cdots u_p\ \star)$ and $C_{v}=(v_1 \cdots v_q\ \star)$, then any shuffle $S$ of the tuples $(u_1, \dots, u_p)$ and $(v_1, \dots, v_q)$ will give rise to a circular order $C_w= (S\ i\ j)$ which is a pre-image of $(C_u, C_v)$. Moreover, all preimages of $(C_u, C_v)$ arise in this way.

Using this, we rewrite
\begin{equation*}
\Omega(\htile{\sdiv}^{(ij)})=  \sum_{(w) \in \Ext_{(ij)}(C_\sdiv)}\PT(w_L)\PT(w_R)=\shuff \sum_{(C_u,C_v)} \PT(u)\PT(v)
\end{equation*}
where the second sum is over $\Ext(C_{\sdiv_L})\times\Ext(C_{\sdiv_R})$.
The second sum can also be written as 
\[\shuff \left(\sum_{(u) \in \Ext(C_{\sdiv_L})} \PT(u) \right)\left(\sum_{(v)\in \Ext(C_{\sdiv_R})} \PT(v) \right) = \shuff \Omega(\htile{\sdiv_L}) \Omega(\htile{\sdiv_R})\]
where the equality uses \cref{cor:w-simp-in-tile-cyclic-order-version}. This shows that \eqref{eq:facet-weight-fcn-factors} holds. 

Now we apply \cref{prop:weight_tile} to \eqref{eq:facet-weight-fcn-factors} to obtain
\[\Omega(\htile{\sdiv}^{(ij)})= \shuff \, (-1)^{k_L+ k_R} \PT(\mathbf{I}_{N_L}) \PT(\mathbf{I}_{N_R}).\]

Now, $k_L+k_R=k-1$ because contracting $i \rightarrow j$ removes the triangles on either side of $i \to j$ in $\sdiv$, and one of these triangles is black and the other is white. This completes the proof of the proposition.
\end{proof}
\begin{remark}

 \cref{eq:facet-weight-fcn-factors} exposes the `factorization' properties of the Parke-Taylor functions in connection with the geometry of tiles, see also \cref{fig:fact_gen}.
\end{remark}

\begin{figure}[h]
\includegraphics[height=1.6in]{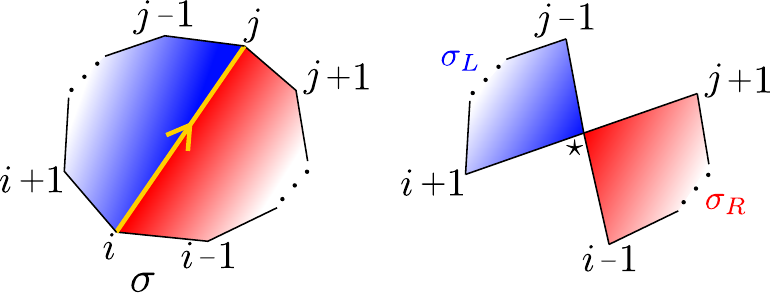}
        \caption{Left: a bicolored subdivision $\sdiv$ with the facet-defining arc $i \rightarrow j$ highlighted. Right: the bicolored subdivisions $\sigma_L$ on $N_L=[i+1,j-1]\cup\{\star\}$ and $\sigma_R$ on $N_R=[j+1,i-1]\cup\{\star\}$ obtained from $\sigma$ by contracting the arc $i \rightarrow j$ and identifying $i$ and $j$ with $\star$. The blue and red colors highlight the respective regions before and after the contraction.}
        \label{fig:fact_gen}
\end{figure}

\begin{example}
    Let $\sigma$ be the bicolored subdivision from \cref{ex:order}. Recall that there are $16$ $w$-simplices contained in $\Gamma_\sigma$, i.e. $|\Ext(C_\sigma)|=16$. The $w$-simplices that have $C_w=(\cdots 1\ 4 \cdots)$ are the following:
    \begin{center}
    \begin{tabular}{|c|c|c|c|c|c|c|}\hline
        \multicolumn{7}{|c|}{$\Ext_{(14)}(C_\sigma)$ } \\ \hline
       $w$ & $5\mathbf{14}236$ & $35\mathbf{14}26$ & $53\mathbf{14}26$ & $235\mathbf{14}6$ & $253\mathbf{14}6$ & $523\mathbf{14}6$ \\ \hline
    \end{tabular}
\end{center}
We have $N_L=\{2,3,\star\}, N_R=\{5,6,\star\}$. 
By \cref{lem:omegafacet}, to compute $\Omega(\Gamma^{(14)}_{\sigma})$, we sum over the weight functions of the $w$-simplices with $(w) \in \Ext_{(14)}(C_\sigma)$. 
Performing the sum gives:
\begin{equation*}
\Omega(\Gamma^{(14)}_{\sigma})=\sum_{(w) \in \Ext_{(14)}(C_\sigma)} \Omega(\simp{w}^{(14)})= 6 \frac{1}{P_{23} P_{3\star} P_{\star 2}}\frac{1}{P_{56 }P_{6 \star} P_{\star 5} }={4 \choose 2} \PT(2,3,\star) \PT(5,6,\star),   
\end{equation*}
which agrees with \cref{prop:weight_facet}, where ${n-2 \choose |N_L|-1}={4 \choose 2}=6$ is the number of shuffles between the sets $\{2,3\},\{5,6\}$. Moreover, $\PT(2,3,\star)=\Omega(\Gamma_{\sigma_L})$ and $\PT(5,6,\star)=\Omega(\Gamma_{\sigma_R})$, where $\sigma_L,\sigma_R$ are as in \cref{fig:ex_fact}.
\end{example}

\begin{figure}[h]
\includegraphics[height=1.0in]{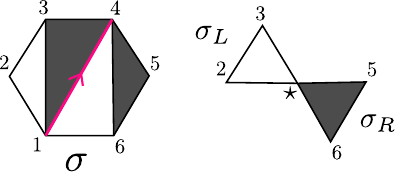}
        \caption{Left: a bicolored subdivision $\sdiv$ with the facet-defining arc $1 \rightarrow 4$ highlighted. Right: the bicolored subdivisions $\sigma_L,\sigma_R$ obtained from $\sigma$ by contracting the arc $1 \rightarrow 4$ and identifying $1$ and $4$ with $\star$.  }
        \label{fig:ex_fact}
\end{figure}

Using \cref{prop:weight_facet} and a similar strategy to the proof of \cref{thm:magic-number}, we can prove the following necessary condition for a collection of tiles to give a tiling.

\begin{theorem}\label{lem:tilesside}
 Let $\mathcal{T}=\{\htile{\sdiv}\}_{\sdiv \in \mathcal{S}}$ be a tiling for $\Delta_{k+1,n}$ and $c \in [1,k]$.
 Let $i \to j$ be an internal arc of $\mathbf{P}_n$. Then
 the number of tiles in $\mathcal{T}$ that have $x_{[i,j-1]}\geq c$ as a facet inequality equals the number of tiles that have $x_{[i,j-1]}\leq c$ as a facet inequality.
 Equivalently, we have 
 \small{ \[\# \{\sdiv \in \mathcal{S}: i \to j \text{ facet-defining, }\area(i \to j)=c\}=\# \{\sdiv \in \mathcal{S}: j \to i \text{ facet-defining, }\area(i \to j)=c-1\}. \]
 }
\end{theorem}

\begin{proof}

Let $H$ denote the hyperplane $\{x_{[i,j-1]}= c\}$ and let $\mathcal{T}_A$, resp. $\mathcal{T}_B$, be the set of tiles in $\mathcal{T}$ with $x_{[i,j-1]}\geq c$, resp. $x_{[i,j-1]}\leq c$, as facet inequality (tiles in $\mathcal{T}_A$ are {above} $H$, while those in $\mathcal{T}_B$ are below). By \cref{thm:hypersimplex-tiles-ineq-facets}, a tile $\htile{\sdiv} \in \mathcal{T}$ is in $\mathcal{T}_A$ precisely when $i \to j$ is a facet-defining arc and $c=\area_\sdiv(i \to j)$; using \cref{thm:hypersimplex-tiles-ineq-facets} and the fact that $x_{[n]}=k+1$ for all points in the hypersimplex, a tile $\htile{\sdiv}$ is in  $\mathcal{T}_B$ precisely when $j \to i$ is a facet-defining arc and $c=\area_\sdiv(i \to j)+1$. This shows that the two statements in the theorem are equivalent.

We would like to show that $|\mathcal{T}_A|= |\mathcal{T}_B|$. Note that if $\htile{\sdiv} \in \mathcal{T}_A$, then the facet $\htile{\sdiv}^{(ij)}$ lies on $H$; if instead $\htile{\sdiv'} \in\mathcal{T}_B$, then the facet $\htile{\sdiv'}^{(ji)}$ lies on $H$.

Consider a facet $\simp{w}^{(ij)}$ of a $w$-simplex, and suppose this facet is contained in the facet $\htile{\sdiv}^{(ij)}$ of a tile $\htile{\sdiv} \in \mathcal{T}_A$. This means $\simp{w}^{(ij)}$ is also contained in $H$. Because $i$ precedes $j$ in $w$, \cref{cor:ineq-and-facet-w-simp} implies $\simp{w}$ in fact lies above $H$. The $w$-simplex $\simp{u}$ which shares the facet $\simp{u}^{(ji)}=\simp{w}^{(ij)}$ with $\simp{w}$ must lie below $H$ and must be contained in a tile $\htile{\sdiv'}$ in the tiling. In particular, $\htile{\sdiv'}$ must be in $\mathcal{T}_B$: $\htile{\sdiv'}$ cannot intersect $\simp{w}$ in its interior, so since $\htile{\sdiv'}$ contains $\simp{u}$, $\htile{\sdiv'}$ must have a facet on $H$ and also be below $H$. The argument exchanging $A$ and $B$ is identical. In summary, $\simp{w}^{(ij)}=\simp{u}^{(ji)}$ is contained in $\htile{\sdiv}^{(ij)}$ for a tile $\htile{\sdiv} \in \mathcal{T}_A$ if and only if it is also contained in $\htile{\sdiv'}^{(ji)}$ for a tile $\htile{\sdiv'} \in \mathcal{T}_B$

This implies
 \begin{equation*}
\sum_{\substack{\simp{w}^{(ij)} \subset \htile{\sdiv}^{(ij)}\\ \text{and } \htile{\sdiv} \in \mathcal{T}_A}} \Omega(\simp{w}^{(ij)})=\sum_{\substack{\simp{u}^{(ji)} \subset \Gamma^{(ji)}_{\sdiv'}\\ \text{and } \Gamma_{\sdiv'} \in \mathcal{T}_B}} \Omega(\simp{u}^{(ji)}).
 \end{equation*}

 The left hand side can be expressed as
 \begin{equation*}
 \sum_{\htile{\sdiv} \in \mathcal{T}_A}\sum_{\simp{w}^{(ij)} \subset \htile{\sdiv}^{(ij)}} \Omega(\simp{w}^{(ij)})=\sum_{\htile{\sdiv} \in \mathcal{T}_A} \Omega(\htile{\sdiv}^{(ij)})=|\mathcal{T}_A| \cdot (-1)^{k-1}\shuff\PT(\mathbf{I}_{N_L})\PT(\mathbf{I}_{N_R})
 \end{equation*}
 where the final equality uses \cref{prop:weight_facet}.
 Analogously, the right hand side is
 \begin{equation*}
 \sum_{\Gamma_{\sdiv'} \in \mathcal{T}_B} \Omega(\Gamma^{(ji)}_{\sdiv'})=|\mathcal{T}_B| \cdot (-1)^{k-1}\shuff\PT(\mathbf{I}_{N_L})\PT(\mathbf{I}_{N_R})
 \end{equation*}
 again using \cref{prop:weight_facet}. (For the tiles in $\mathcal{T}_B$, $j \to i$ is a facet-defining arc rather than $i \to j$, but the expression in \cref{prop:weight_facet} is symmetric with respect to $i$ and $j$, so we obtain the same expression.)
 This implies $|\mathcal{T}_A|=|\mathcal{T}_B|$.
 \end{proof}

\begin{example}
Consider the collection $\mathcal{S}$ of bicolored subdivisions in \cref{fig:tiling}. There are two bicolored subdivisions, \textcircled{\raisebox{-0.9pt}{2}} and \textcircled{\raisebox{-0.9pt}{5}}
, that have $4 \rightarrow 7$ as facet-defining arc and $\area(4 \to 7)=2$. Correspondingly, \textcircled{\raisebox{-0.9pt}{3}} and \textcircled{\raisebox{-0.9pt}{4}}
 have $7 \rightarrow 4$ as facet-defining arc and $\area(4 \to 7)=2-1=1$. There is one bicolored subdivision, \textcircled{\raisebox{-0.9pt}{7}}
, that has $4 \rightarrow 7$ as facet-defining arc and $\area(4 \to 7)=1$. Correspondingly, there is one bicolored subdivision, \textcircled{\raisebox{-0.9pt}{6}}
, that has $7 \rightarrow 4$ as facet-defining arc and $\area(4 \to 7)=1-1=0$. This verifies $\mathcal{S}$ satisfies \cref{lem:tilesside} for the arc $4 \to 7$.  
\end{example}

\begin{corollary}\label{cor:necessary}
 Suppose that a collection $\mathcal{S}$ of bicolored subdivisions indexes a positroid tiling of $\Delta_{k+1, n}$
 (or an all-$Z$ tiling of $\mathcal{A}_{n,k,2})$. Then for every diagonal $(i,j)$ of the $n$-gon, 
 \[\#\{\sdiv \in \mathcal{S}: i \to j \text{ is a facet-defining arc}\}=\#\{\sdiv \in \mathcal{S}: j \to i \text{ is a facet-defining arc}\}. \]
\end{corollary}

\begin{proof}
Fix a diagonal $(i,j)$ of the $n$-gon. Let $\mathcal{S}_A^{(r)}$ (resp. $\mathcal{S}_B^{(r)}$) be the subdivisions in $\mathcal{S}$ where $i \to j$ (resp. $j \to i$) is a facet-defining arc and $\area(i \to j)=r$ (resp. $\area(i \to j)=r$). By \cref{lem:tilesside}, we have $|\mathcal{S}_A^{(r)}|=|\mathcal{S}_B^{(r-1)}|$. The number of subdivisions in $\mathcal{S}$ with $i \to j$ as a facet-defining arc is 
\[\sum_{r=1}^k |\mathcal{S}_A^{(r)}| = \sum_{r=0}^{k-1}|\mathcal{S}_B^{(r)}|.\]
The right hand side is the number of $\sdiv \in \mathcal{S}$ with $j \to i$ as a facet-defining arc.

\begin{example}
Consider the collection $\mathcal{S}$ of bicolored subdivisions in \cref{fig:tiling}. There are $3$ bicolored subdivisions ($\textcircled{\raisebox{-0.9pt}{2}},\textcircled{\raisebox{-0.9pt}{5}},\textcircled{\raisebox{-0.9pt}{7}}
$) that have $4 \rightarrow 7$ as facet-defining arc. Correspondingly, there are $3$ bicolored subdivisions ($\textcircled{\raisebox{-0.9pt}{3}},\textcircled{\raisebox{-0.9pt}{4}}
,\textcircled{\raisebox{-0.9pt}{6}}$) that have $7 \rightarrow 4$ as facet-defining arc. Hence, $\mathcal{S}$ satisfies \cref{cor:necessary}.  
\end{example}
\end{proof}

Using \cref{cor:necessary}, we obtain another combinatorial necessary condition for a collection of bicolored subdivisions to index a tiling.

\begin{proposition}\label{prop:covering-ngon-necessary}
Suppose that a collection $\mathcal{S}$ of bicolored subdivisions indexes a positroid tiling of $\Delta_{k+1, n}$ (or an 
all-$Z$ tiling of $\mathcal{A}_{n,k,2}$).
     Superimposing all bicolored subdivisions in $\mathcal{S}$, we obtain a $\binom{n-3}{k-1}$-fold covering of the $n$-gon by black polygons. 
\end{proposition}

\begin{example}
Consider the collection $\mathcal{S}$ of bicolored subdivisions in \cref{fig:tiling}. Let us color its black polygons into $6$ different colors: red ($\textcircled{\raisebox{-0.9pt}{A}}$), orange ($\textcircled{\raisebox{-0.9pt}{B}}$), yellow ($\textcircled{\raisebox{-0.9pt}{C}}$), green ($\textcircled{\raisebox{-0.9pt}{D}}$), blue ($\textcircled{\raisebox{-0.9pt}{E}}$) and purple ($\textcircled{\raisebox{-0.9pt}{F}}$) as in \cref{fig:overlap_tilings}. Then each collection of polygons of the same color is a subdivision of the $7$-gon. Hence superimposing all bicolored subdivisions in $\mathcal{S}$, we obtain a ${7-3 \choose 3-1}={4 \choose 2}=6$-fold covering of the $7$-gon by black polygons.

    \begin{figure}[h]
\includegraphics[width=0.9\textwidth]{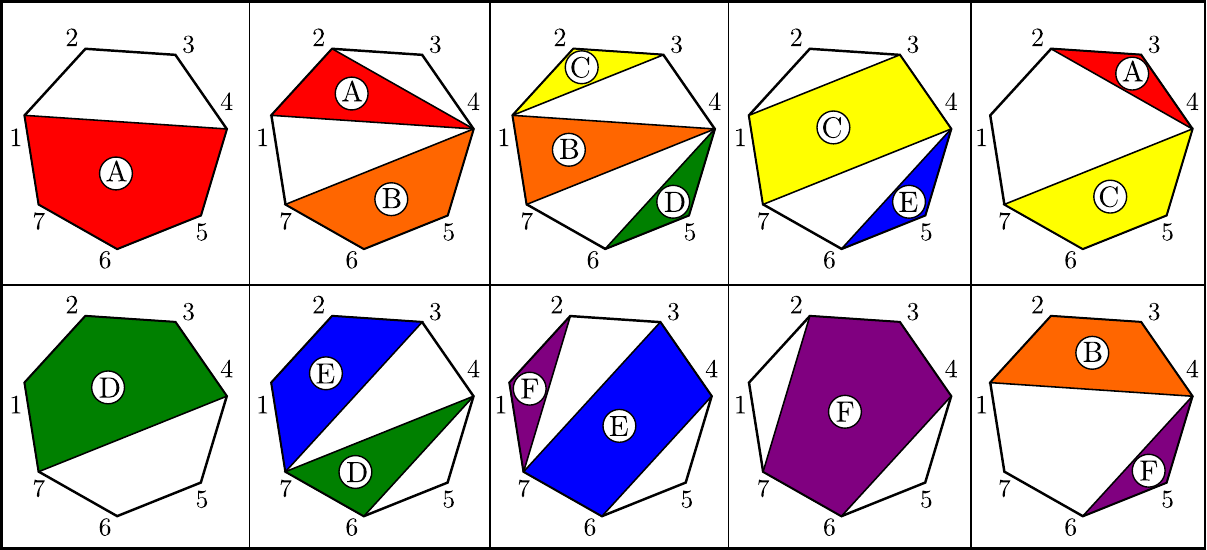}
\caption{Partitioning the black polygons in $\mathcal{S}$ of \cref{fig:tiling}  into $6$ subdivisions of the $7$-gon.}
        \label{fig:overlap_tilings}
\end{figure}
\end{example}

\begin{proof}[Proof of \cref{prop:covering-ngon-necessary}]
Let $\mathcal{B}$ be the multi-set of black polygons appearing in the subdivisions in $\mathcal{S}$. Since each subdivision in $\mathcal{S}$ has area $k$ and by \cref{thm:magic-number} we have $|\mathcal{S}|= \binom{n-2}{k}$, the total area of all polygons in $\mathcal{B}$ is $k \binom{n-2}{k}$.

Similarly to bicolored subdivisions, we call an arc $i \to j$ a facet-defining arc of a polygon $p$ in $\mathcal{B}$ if $(i,j)$ is an edge of $t$ and the other vertices of $p$ is in the cyclic interval $[i+1,j-1]$. 
\cref{cor:necessary} implies that the multiset $\mathcal{B}$ has the property that for all diagonals $(i,j)$ of the $n$-gon,
\begin{equation}\label{eq:arcs-covered-evenly}
\#\{p \in \mathcal{B}:i \to j \text{ a facet-defining arc} \}= \#\{p \in \mathcal{B}:j \to i \text{ a facet-defining arc} \}.\end{equation}

Now, choose some $p \in \mathcal{B}$. For each internal facet-defining arc $i \to j$ of $p$, \eqref{eq:arcs-covered-evenly} implies there is another polygon $p_{ij} \in \mathcal{B}$ for which $j \to i$ is a facet defining arc--meaning that $p_{ij}$ lies on the opposite side of the diagonal $(i,j)$. The polygons $p_{ij}$ have disjoint interiors. For each $p_{ij}$, consider each ``exposed" edge; that is, each internal facet-defining arc which is not an edge of $p$. Again, \eqref{eq:arcs-covered-evenly} implies that we can find a polygon in $\mathcal{B}$ which covers the other side of this edge. Again, all of the polygons we choose to cover the exposed edges of the $p_{ij}$ have disjoint interiors. We continue this process, successively covering the ``exposed" facet-defining arcs of the polygons we have obtained, until we have a collection of polygons in $\mathcal{B}$ which give a subdivision of the $n$-gon. Let $\mathcal{B}'$ be the multiset obtained from $\mathcal{B}$ by removing this collection of polygons. Note that $\mathcal{B}'$ still satisfies \eqref{eq:arcs-covered-evenly}. So we may iterate this procedure of choosing a polygon in $\mathcal{B'}$, then successively choosing polygons to cover the exposed edges until we obtain a subdivision, then removing the subdivision from $\mathcal{B'}$ to obtain a new multiset. This procedure will terminate when our multi-set becomes empty; that is, when we have removed $k \binom{n-2}{k}$ total area from $\mathcal{B}$.

Each time we iterate, we remove a subdivision of the $n$-gon, and thus remove polygons with total area $n-2$. So the procedure terminates after $\binom{n-3}{k-1}$ steps, since $$(n-2)\binom{n-3}{k-1}= k\binom{n-2}{k}.$$ 

In summary, the black polygons from the subdivisions in $\mathcal{S}$ cover the $n$-gon $\binom{n-3}{k-1}$ times, as desired.
    
\end{proof}

\begin{remark}
    The proof of \cref{prop:covering-ngon-necessary} shows something slightly stronger. In the setting of \cref{prop:covering-ngon-necessary}, one can partition the multi-set of black polygons in $\mathcal{S}$ into $\binom{n-3}{k-1}$ subdivisions of the $n$-gon. See \cref{fig:overlap_tilings} for an example.
\end{remark}

\section{Combinatorics of Parke-Taylor polytopes and Parke-Taylor identities}\label{sec:combPT}

In this section we use tricolored subdivisions (cf. \cref{def:tricolored}) to introduce and study Parke-Taylor polytopes.  We also 
use them to derive new Parke-Taylor identities.

\subsection{Parke-Taylor polytopes}

Our combinatorial results on polytopes
associated to bicolored subdivisions in \cref{sec:combtiles} have natural extensions to 
tricolored subdivisions as well.  However, in order to make sense of 
this, we need to work with the projection $\pi: \R^n \to \R^{n-1}$ from \cref{rem:hypercube}.

\begin{definition}
Given a bicolored subdivision $\sigma$, let 
$\ptile{\sigma}$ denote the projected polytope $\pi(\htile{\sigma})$.
Given a $w$-simplex $\Delta_w$, let $\psimp_w$ denote
the projected simplex $\pi(\Delta_w)$.
\end{definition}

We also need to extend \cref{def:bicolored} to the 
setting of tricolored subdivisions.

\begin{definition}
Let $\tdiv$ be a tricolored subdivision.
      Given a pair of vertices $i,j$ of $\mathbf{P}_n$, we say that 
       the arc $i \to j$ is \emph{compatible} with $\tdiv$ if the arc either 
       is an edge of a black, white, or grey polygon, 
       or lies  entirely inside a black or white polygon of $\tdiv$.  
If $i \to j$ is compatible with $\tdiv$, the \emph{area to the left of }$i \to j$ (respectively, \emph{grey area to the left of }$i \to j$), denoted by $\area(i \to j)$ (respectively, $\garea(i \to j)$), is the number of black triangles (respectively, grey triangles) to the left of $i \to j$ in any triangulation of the black (respectively, grey) polygons of $\tdiv.$  
\end{definition}

\begin{definition}
Let $\tdiv$ be a tricolored subdivision of $\mathbf{P}_n$ of type 
$(k,\ell,n).$
We define the 
\emph{Parke-Taylor}
polytope $\ptile{\tdiv} \subset \R^{n-1}$ by the following inequalities: 
    for any compatible arc $i \to j$ with $i<j, $
    $$ \area(i\to j) \leq x_{[i,j-1]}\leq \area(i\to j)+\garea(i\to j)+1.$$
\end{definition}

\begin{example}\label{ex:cube}
    If $\tdiv$ is the tricolored subdivision of $\mathbf{P}_n$ which is just a grey polygon on $n$ vertices, then the Parke-Taylor polytope
    $\ptile{\tdiv}$ is the unit hypercube $\mbox{\mancube}_{n-1} \subset \R^{n-1}$.
\end{example}

\begin{example}
    For the tricolored subdivision $\tdiv$ from \cref{fig:tricolored} of type $(3,2,8)$, the Parke-Taylor  polytope $\ptile{\tdiv} \subset \mathbb{R}^7$ is defined by the inequalities $0 \leq x_i \leq 1$ for $1 \leq i\leq 7$, \ 
     $3 \leq x_{[1,7]} \leq 6$, and 
    \begin{align*}
        1 \leq x_{[5,6]} \leq 2, \qquad 0 \leq x_{[2,4]} \leq 2, \qquad 1 \leq x_{[2,6]} \leq 4\\
        2 \leq x_{[1,6]} \leq 5, \qquad 2 \leq x_{[2,7]} \leq 5.
    \end{align*}
     Note that this is not a minimal description of $\ptile{\tdiv}$, as the inequality $1 \leq x_{[2,6]} \leq 4$ is implied by 
$1 \leq x_{[5,6]} \leq 2$ and $0 \leq x_{[2,4]} \leq 2$.     
\end{example}

The Parke-Taylor polytope $\ptile{\tdiv}$ is related to the (projected) positroid tiles as follows.

\begin{proposition}\label{prop:tri-sub-polytope-as-union} 
Let $\tdiv$ be a tricolored subdivision. Let $\mathcal{S}$ denote the set of bicolored subdivisions obtained from $\tdiv$ by replacing each grey polygon $p$ by any bicolored subdivision $\sdiv_p$ of $p$. Then we have a covering of 
$\ptile{\tdiv}$ by projected positroid tiles:
\begin{equation} \label{eq:bigUnion}
\ptile{\tdiv} = \bigcup_{\sdiv \in \mathcal{S}} \ptile{\sdiv}.
\end{equation}

Moreover, there is a collection $\mathcal{K}$ of bicolored subdivisions obtained from 
$\tdiv$ by choosing a distinguished vertex $v_p$ of each grey polygon $p$ of $\tdiv$ and replacing $p$ with all possible kermit subdivisions of $p$ (based at $v_p$), such that we have a covering of 
$\ptile{\tdiv}$ by projected positroid tiles with disjoint interiors:
\begin{equation} \label{eq:Kermit}
\ptile{\tdiv} = \bigcup_{\sdiv \in \mathcal{K}} \ptile{\sdiv}.
\end{equation}

\end{proposition}
\begin{remark}
\cref{eq:Kermit} should remain true if we replace $\mathcal{K}$ by any collection $\mathcal{K}'$ of bicolored subdivisions obtained as follows: for each grey $r$-gon $p$ of $\tdiv$, choose a tiling $\mathcal{T}_a =\{\htile{\sdiv'}\}_{\sdiv' \in \mathcal{S}_a}$ of $\Delta_{a+1, r}$ for $a=0, \dots, r-2$. A bicolored subdivision $\sdiv$ is in $\mathcal{K'}$ if it is obtained from $\tdiv$ by replacing each $p$ with some $\sdiv' \in \mathcal{S}_a$. In this language, to obtain the collection $\mathcal{K}$ in \cref{prop:tri-sub-polytope-as-union}, choose each $\mathcal{T}_a$ to be a kermit tiling based at $v_p$. 
\end{remark}

See \cref{fig:pt-polytope-subdivision} for an example of the set $\mathcal{K}$.

\begin{proof}[Proof of \cref{prop:tri-sub-polytope-as-union}]
We first show 
$\bigcup_{\sdiv \in \mathcal{S}} \ptile{\sdiv} \subset \ptile{\tdiv}.$

    Let $\sdiv \in \mathcal{S}$ be a bicolored subdivision, and consider a point $x \in \ptile{\sdiv}.$ Choose an arc $i \to j$ (with $i <j$) that is compatible with $\tdiv$. 
    Clearly 
    the arc $i \to j$ is also compatible with $\sdiv$. We have that $x$ satisfies
    \[\area_{\sdiv}(i \to j) \leq x_{[i, j-1]} \leq \area_{\sdiv}(i \to j)+1.\]
    Note that $\area_{\tdiv}(i \to j) \leq \area_{\sdiv}(i \to j)$, since $\sdiv$ can only have more black polygons than $\tdiv$, and $$\area_{\tdiv}(i \to j) + \garea_{\tdiv}(i \to j) \geq \area_{\sdiv}(i \to j),$$ since at worst all grey regions to the left of $i \to j$ in $\tdiv$ have been colored black in $\sdiv$. This shows that $\bigcup_{\sdiv \in \mathcal{S}} \ptile{\sdiv} \subset \ptile{\tdiv}$.

    Now we will construct the collection $\mathcal{K}$ and show $\ptile{\tdiv} \subset \bigcup_{\sdiv \in \mathcal{K}} \ptile{\sdiv}$. To construct $\mathcal{K}$, we will choose distinguished vertices of the grey polygons of $\tdiv$ oen by one.

     Pick a grey polygon $p$ of $\tdiv$ 
     with no grey polygons to its left; that is, some boundary arc $i \to j$ of $p$ (with $i<j$) has $p$ to its left but no other grey polygon. Choose $i$ as the distinguished vertex $v_p$ of $p$. Repeat this process, choosing a grey polygon $p'$ of $\tdiv$ such that for some boundary arc $a \to b$ of $p'$ (with $a<b$), the grey polygons to the left of $a \to b$ are $p'$ and grey polygons whose distinguished vertices have already been chosen. Choose $a$ as the distinguished vertex $v_{p'}$.
    
    Now, consider a point $x \in \ptile{\tdiv}$ whose coordinate sums $x_{[i,j]}$ are all non-integral. Let $r := \lfloor x_{[n-1]}\rfloor$. Lift $x$ to the point $y=(x, r+1-x_{[n-1]}) \in \Delta_{r+1,n} \subset \R^n$. We would like to find a bicolored subdivision $\sdiv \in \mathcal{K}$ of area $r$ such that $y \in \htile{\sdiv}$. To do so, we will color the grey polygons of $\tdiv$ one by one, in the same order as we chose the distinguished vertices. 
    
    We begin with grey polygon $p$, with distinguished vertex $v_p=i$ and boundary arc $i \to j$ with no other grey polygons to its left. Triangulate $p$ with edges with one endpoint $v_p$. Once we color each triangle black or white, we will obtain a kermit subdivision of $p$ based at $v_p$.

     Suppose that going clockwise around $p$, we see vertex $v_p=i$, then $a$ then $b$, so $i \to b$ is an arc in the triangulation. Note that $i \to a$ and $a \to b$ are compatible with $\tdiv$ and
    $$\area_\tdiv(i \to b) = \area_\tdiv(i \to a)+ \area_\tdiv(a \to b).$$
    Meanwhile the inequalities of $\ptile{\tdiv}$ for arcs $i \to a$ and $a \to b$ say that 
    \begin{align*}
     \area_\tdiv(i \to a) \leq y_{[i,a-1]} &\leq \area_{\tdiv}(i\to a)+1, \text{ and } \\
     \area_\tdiv(a \to b) \leq y_{[a,b-1]} &\leq \area_{\tdiv}(a\to b)+1, 
    \end{align*}
    and hence 
    $y$ satisfies
    \[\area_\tdiv(i \to b) \leq y_{[i, b-1]} \leq \area_\tdiv(i \to b) +2.\]
    
    By assumption $y$ has non-integral consecutive coordinate sums, so we have either 
    \[\area_\tdiv(i \to b) < y_{[i, b-1]}< \area_\tdiv(i \to b)+ 1 \quad \text{ or }\quad \area_\tdiv(i \to b)+1 < y_{[i, b-1]}< \area_\tdiv(i \to b)+ 2.\]
    In the former case, we color the triangle on $i,a, b$ white in $\sdiv$ and in the latter case we color the triangle black. In either case, $y$ satisfies
    \[\area_{\sdiv}(i \to b) < y_{[i, b-1]}< \area_{\sdiv}(i \to b)+ 1.\]
    Repeat this for all triangles of the chosen triangulation for $p$, starting with the triangle adjacent to the one we just colored. This gives a new tricolored subdivision $\tdiv'$ with one fewer grey polygon than $\tdiv$. Repeating the above procedure on $\tdiv'$ produces the desired bicolored subdivision $\sdiv$.

    This shows that a dense subset of points in $\ptile{\tdiv}$ is contained in $\bigcup_{\sdiv \in \mathcal{K}}\ptile{\sdiv}$. Since this union is closed, the desired containment follows.

    Now, we have
    \[\bigcup_{\sdiv \in \mathcal{K}} \ptile{\sdiv} \subset \bigcup_{\sdiv \in \mathcal{S}} \ptile{\sdiv} \subset \ptile{\tdiv} \subset \bigcup_{\sdiv \in \mathcal{K}} \ptile{\sdiv}\]
    where the leftmost containment is because $\mathcal{K} \subset \mathcal{S}$ and the other containments have been argued above. This shows both \eqref{eq:bigUnion} and \eqref{eq:Kermit}.

    Finally, the polytopes $\{\ptile{\sdiv}\}_{\sdiv \in \mathcal{K}}$ have disjoint interiors. Indeed, it follows from \cref{cor:allk} that a projected $w$-simplex $\psimp_{w}$ lies in at most one $\ptile{\sdiv}$ for $\sdiv \in \mathcal{K}$. The polytopes $\ptile{\sdiv}$ are triangulated by projected $w$-simplices, so this implies no two polytopes in $\{\ptile{\sdiv}\}_{\sdiv \in \mathcal{K}}$ have full-dimensional intersection and thus they have disjoint interiors.
\end{proof}

We also have a triangulation of $\ptile{\tdiv}$ into projected 
$w$-simplices; this is a generalization of \cref{cor:w-simp-in-tile-cyclic-order-version}.

\begin{figure}
    \centering
    \includegraphics[width=\textwidth]{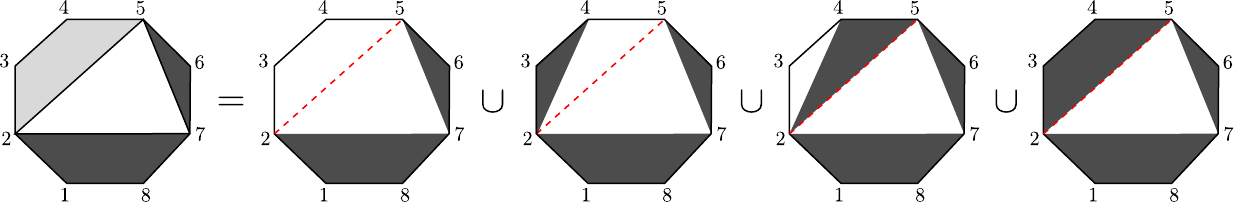}
    \caption{An illustration of the second part of \cref{prop:tri-sub-polytope-as-union}. The polytope $\ptile{\sdiv}$ whose $\sdiv$ is on the left hand side and in \cref{fig:tricolored} is the union of the polytopes $\ptile{\tau}$ for $\tau$ on the right hand side. The distinguished vertex of the grey polygon is $2$ in this case.}
    \label{fig:pt-polytope-subdivision}
\end{figure}

\begin{proposition}\label{prop:volume}
Let $\tdiv$ be a tricolored subdivision of type $(k,\ell,n)$. Then
\[\ptile{\tdiv}= \bigcup_{(w) \in \Ext(C_\tdiv)} \psimp_w.\]
It follows that the volume of $\ptile{\tdiv}$ is the number of circular extensions
of the cyclic order $C_{\tdiv}$. That is, $\Vol(\htile{\tdiv})= |\Ext(C_\tdiv)|$.
\end{proposition}
\begin{proof}
By \cref{prop:tri-sub-polytope-as-union},
$$\ptile{\tdiv}= \bigcup_{\sdiv} \ptile{\sdiv}$$
where the union is over the
collection $\mathcal{K}$ of bicolored subdivisions obtained from 
$\tdiv$ by choosing a distinguished vertex $v_p$ of each grey polygon $p$ of $\tdiv$ and replacing $p$ with all possible kermit subdivisions of $p$ (based at $v_p$).
Therefore if $\psimp_w$ lies in $\ptile{\tdiv}$, then
$\psimp_w$ lies in some $\ptile{\sdiv}$
for some $\sdiv\in \mathcal{K}$, which implies that $C_w$ is a circular extension of $C_{\sdiv}$.  But $C_{\tdiv} \subset C_{\sdiv}$, and hence $C_w$ is a circular extension of $C_{\tdiv}$ as well.

Given a circular extension $C_w$ of $C_{\tdiv}$, let $C_w|_p$ be the restriction of $C_w$ to the vertices of a grey polygon $p$ of $\tdiv$. By \cref{cor:allk}, there is exactly one kermit subdivision $\sdiv(p)$ of $p$ (using the distinguished vertex $v_p$) for which $C_w|_p$ is a circular extension of $C_{\sdiv(p)}$. This means that $C_w$ is a circular extension of $C_{\sdiv}$, where $\sdiv$ is obtained by subdividing each grey polygon $p$ with $\sdiv(p)$. Such a subdivision $\sdiv$ is one of the subdivisions described in the previous paragraph.
\end{proof}

\begin{example}
    For the tricolored subdivision $\tdiv$ in \cref{fig:tricolored}, we have 
    \[\Vol(\ptile{\tdiv})=55+62+127+127=371\]
    which is the number of circular extensions of the cyclic order 
    $C_\tdiv=C_{(2,5,7)} \cup C_{(5,7,6)} \cup C_{(1,8,7,2)}.$
    Each term in the sum above is the volume of $\ptile{\sdiv}$ for $\sdiv$ on the right hand side of \cref{fig:pt-polytope-subdivision}.
\end{example}

\begin{remark}\label{rem:AJVR}
\cref{prop:volume} is quite analogous to 
a result from \cite{AJR_cyclicOrders}, 
which shows that the (normalized) volume of a certain polytope $B_{I,n}$ equals the number of circular extensions of a partial cyclic order $A_{I,n}$.  Here $I$ is a subset of 
$\binom{[n]}{2}$,  the polytope $B_{I,n}$ is defined as 
$$B_{I,n} = \{(x_1,\dots,x_n)\in [0,1]^n \ \vert \ x_{[i+1,j]} \leq 1 \text{ for each }\{i<j\} \in I\},$$ 
and the cyclic order $A_{I,n}$ is 
defined as the union of the chains $C_{(i, i+1,\dots,j)}$ for each $\{i<j\}\in I$.
In the special case when the intervals
$[i,j]$ (for $\{i<j\} \in I$) are noncrossing (i.e. we do not have 
$i<i'<j<j'$), $A_{I,n}$ is equal to $C_{\tdiv}$, where $\tdiv$ is a tricolored subdivision with no black polygons such that the vertices of every white polygon form an interval $[i,j]$; and $B_{I,n} \subset \R^n$ is, up to renaming variables, equal to $\ptile{\tdiv}$ times the interval $[0,1]$. For other $I\subset \binom{[n]}{2}$ and for other subdivisions $\tdiv$, the cyclic orders $A_{I,n}$ and $C_{\tdiv}$, and the polytopes $B_{I,n}$ and $\ptile{\tdiv}$, are not obviously related.
\end{remark}

\subsection{Parke-Taylor identities}

In this section we derive identities for Parke--Taylor functions, one for each tricolored subdivision.
Parke-Taylor identities have previously been studied and connected to: non-planar generalizations of \emph{plabic graphs} \cite{Arkani-Hamed:2014bca, Cachazo:2018wvl}; the \emph{momentum amplituhedron} \cite{Damgaard:2021qbi}; \emph{Lie polynomials} \cite{Frost:2019fjn}; and \emph{CGEM amplitudes} in relation to the configuration of $n$ points in $\mathbb{CP}^{k-1}$ \cite{Cachazo:2023ltw}.

\begin{theorem}\label{thm:gen-parke-taylor}
Let $\tdiv$ be a tricolored subdivision of $\mathbf{P}_n$ which contains at least one grey
polygon, and let $C_{\tdiv}$ be the corresponding
cyclic partial order.  Recall that 
$\Ext(C_{\tdiv})$ is the set of cyclic extensions of $C_{\tdiv}$.
 Then we have that 
$$\sum_{(w)\in \Ext(C_{\tdiv})} \PT(w) = 0.$$
\end{theorem}

\begin{proof}

\begin{figure}[h]
\includegraphics[height=2in]{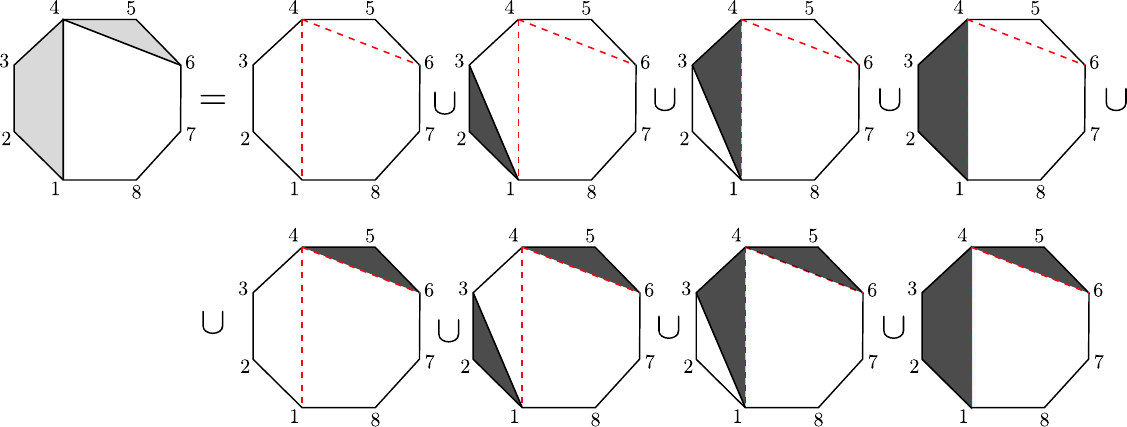}
        \caption{Example for both \cref{thm:gen-parke-taylor} and \cref{thm2:gen-parke-taylor} (with $S=\{1, 4, 6, 7, 8\}$). The dashed lines are to help the reader and are not part of the tricolored subdivision.}
        \label{fig:bigexample}
\end{figure}

Let us label the grey polygons of $\tdiv$ by $P_1,\dots,P_{\ell}$, and for each one, 
choose a distinguished vertex $v_1,\dots,v_{\ell}.$   We will apply \cref{cor:allk} and \cref{rem:subcyclic} (see also \cref{fig:kermitsubdivisions}) to each grey polygon $P_i$.  In particular, each total cyclic order
on the vertices of $P_i$ must be a cyclic extension of one of the partial cyclic orders associated to a ``kermit'' bicolored subdivision (based at $v_i$) of $P_i$.

Let $\sigma_1,\dots,\sigma_r$ be the set of all bicolored subdivisions of $\mathbf{P}_n$ obtained
from $\tdiv$ by replacing each grey polygon $P_i$ (for $1\leq i \leq \ell$) by a bicolored kermit subdivision of that polygon based at $v_i$, 
see \cref{fig:bigexample}.
By \cref{cor:allk}, each circular extension $C_w$ of $C_{\tdiv}$ must be a circular extension
of exactly one of the partial cyclic orders $C_{\sigma_i}$ for $1\leq i \leq r,$ which by
 \cref{cor:w-simp-in-tile-cyclic-order-version} corresponds to a $w$-simplex from the corresponding tile 
 $\Gamma_{\sigma_i}.$  Thus we have 
\begin{equation*}
    \sum_{(w)\in \Ext(C_{\tdiv})} \PT(w) =\sum_{i=1}^r \Omega(\Gamma_{\sigma_i}).
\end{equation*}
But by \cref{prop:weight_tile}, the Parke-Taylor function of a tile $\Gamma_{\sigma_i}$ is $(-1)^k \PT(\mathbf{I}_n)$, where $k$ is the number of black triangles in any triangulation of the black polygons of $\sigma_i$.
Therefore, if $d_i$ is the number of vertices of $P_i$ and $k'$ is the area of $\tdiv$, we get 
\begin{align*}
    \sum_{(w)\in \Ext(C_\tdiv)} \PT(w) &=\sum_{i=1}^r \Omega(\Gamma_{\sigma_i})= (-1)^{k'} \PT(\mathbf{I}_n) \cdot\prod_{i=1}^{\ell} \left(\sum_{k_i=0}^{d_i-2} {d_i-2 \choose k_i} (-1)^{k_i}\right)=0,
\end{align*}
where the last equality holds because each term in the product is 0.
\end{proof}

\begin{figure}[h]
\includegraphics[height=1.1in]{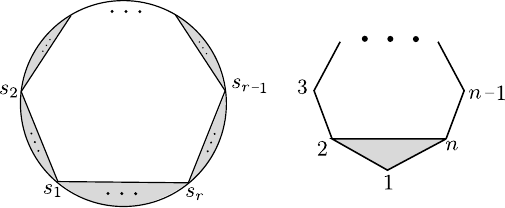}
        \caption{Tricolored subdivisions generating the relation in \cref{thm2:gen-parke-taylor} (left) and $U(1)$ decoupling identities as a special case (right).}
        \label{fig:pt_ids}
\end{figure}

\begin{remark}\label{rk:weight_PTpolytope} We can extend \cref{def:weight_pos_pol} and define the weight function $\Omega(\ptile{\tdiv})$ of a Parke-Taylor polytope $\ptile{\tdiv}$ as 
\begin{equation*}
\Omega(\ptile{\tdiv}):=\sum_{\tilde{\Delta}_{(w)}\subset \ptile{\tdiv}} \Omega(\simp{w})=\sum_{(w) \in \Ext(C_\tau)} \PT(w),
\end{equation*}
where the second equality follows from \cref{prop:volume}. \cref{thm:gen-parke-taylor} can then be restated as: the weight function of a Parke-Taylor polytope $\ptile{\tdiv}$ vanishes if $\tau$ has at least one grey polygon.
    
\end{remark}

\begin{remark}\label{rmk:Sn-action}
    There is an $S_n$ action on $\Gr_{2,n}$, where the permutation $\pi$ acts by sending the matrix $A$ with columns $v_1, \dots, v_n$ to the matrix $\pi(A)$ with columns $v_{\pi(1)}, \dots, v_{\pi(n)}.$ We have that $\PT(w)$ evaluated on $\pi(A)$ is equal to $\PT(\pi w)$ evaluated on $A$. Thus, given any identity of Parke-Taylor functions $\PT(w)$, we can obtain another identity by left-multiplying each argument $w$ by a fixed permutation in $S_n$.
\end{remark}

We highlight a few special cases of \cref{thm:gen-parke-taylor}. 

If we consider a tricolored subdivision with $\mathbf{P}_n$ all colored grey, then by \cref{ex:cube}, \cref{thm:gen-parke-taylor} and \cref{rk:weight_PTpolytope}, we have the following.
\begin{corollary}
The weight function of the
unit hypercube $\mbox{\mancube}_{n-1}$ is
\begin{equation*}
    \Omega(\mbox{\mancube}_{n-1})=\sum_{k=0}^{n-2} \Omega(\Delta_{k+1,n})=\sum_{w \in D_n} \Omega(\simp{w})=0.
\end{equation*}
Equivalently, the sum over $w \in D_n$ of the weight functions of $w$-simplices is zero.
\end{corollary}

If we consider a tricolored subdivision $\tdiv$ with one white polygon and all others grey (see \cref{fig:pt_ids}), we obtain the following.

\begin{corollary}\label{thm2:gen-parke-taylor}
Let $S=\{s_1< \dots < s_r\} \subsetneq [n-1]$, and let 
$\mathcal{D}_S$ be the set of permutations in $D_n$ in which 
$s_1,\dots, s_r$ appear in order.  Then we have that 
$$\sum_{w\in \mathcal{D}_S} \PT(w) = 0.$$
\end{corollary}
\begin{proof}
Consider the $n$-gon $\mathbf{P}_n$ with vertices $1,2,\dots,n$ and let 
$R$ be the subpolygon with vertices $S \cup \{n\}.$  
Let $\tdiv$ be the tricolored subdivision 
of $\mathbf{P}_n$ in which the subpolygon $R$ is white and the rest of $\mathbf{P}_n$ is grey, see \cref{fig:bigexample}.
By \cref{def:cyclic-from-perm-subdiv}, $\mathcal{D}_S$ exactly corresponds to the 
total cyclic orders which are circular extensions of the circular order $C_{\tdiv}.$
   Now the result follows from \cref{thm:gen-parke-taylor}. 
\end{proof}

In the specific case of \cref{thm2:gen-parke-taylor} when $r=n-1$ and $S=[n-1]\setminus \{1\}$, we obtain the known \emph{$U(1)$ decoupling identities}\footnote{This relation has the interpretation of a scattering amplitude of $n-1$ $SU(N)$ `gluons' and a $U(1)$ `photon' that vanishes because the latter is decoupled from (i.e. does not interact with) the former.} for Parke-Taylor functions.
\begin{corollary}\label{cor:U1} We have
\begin{equation*}
\PT(1,2,3,\ldots,n)+\PT(2,1,3,\ldots,n)+\PT(2,3,1,\ldots,n)+\ldots+\PT(2,3,\ldots,1,n)=0. 
\end{equation*}
    
\end{corollary}

We end this section with a variation of \cref{thm2:gen-parke-taylor}, which gives the \emph{shuffle identities} for Parke-Taylor functions (see \cite[Section 3.1]{Frost:2019fjn} and \cite{Cresson}). 
\begin{proposition} \label{prop:shufflePT}
Let $I\subseteq [n-1]$ and fix a permutation $u=u_1 \dots u_r$ of $I$ and a permutation $v=v_1 \dots v_{n-r-1}$ of $[n-1]\setminus I$. Let $\Sh_n(u,v)$ be the set of permutations in $D_n$ in which $u_1,\ldots,u_r$ appear in order and $v_1,\ldots,v_{n-r-1}$ appear in order. Then we have
\begin{equation*}
 \sum_{w\in \Sh_n(u,v)} \PT(w) = 0.   
\end{equation*}
\begin{proof}
Let $\tau$ be the tricolored subdivision consisting of a black polygon with vertices $[r] \cup \{n\}$, a white polygon with vertices $[r+1, n]$, and a grey polygon with vertices $\{r,r+1,n\}$. Let $C_\tdiv$ be the corresponding partial cyclic order. The set of circular extensions $\Ext(C_{\tdiv})$ consists of circular orders $(w)$ in which $n, r, r-1, \dots, 1$ appear in that order clockwise and $r+1, \dots, n$ appear in that order clockwise. Rotating so $n$ is at the end, we have that $(w)$ corresponds to $w \in D_n$ where $r, r-1, \dots, 1$ appear in order and $r+1, \dots, n-1$ appear in order. Define the permutation $x= r~r-1 \dots 1$ on $[r]$ and the permutation $y= r+1 ~ r+2 \dots n-1$ on $[r, n-1]$. By \cref{thm:gen-parke-taylor} and the above paragraph, we have
\begin{equation}\label{eq:shuffleproof}
 \sum_{w\in \Sh_n(x,y)} \PT(w) =\sum_{(w)\in \Ext(C_{\tdiv})} \PT(w) = 0.   
\end{equation}

Now, we define a permutation $\pi$ on $[n]$ so that left multiplication by $\pi$ sends $\Sh_n(x,y)$ to $\Sh_n(u,v)$. Writing $\pi$ in two-line notation, 
\[\pi:= \begin{pmatrix}
    1 & 2 & \cdots & r & r+1 & r+2 & \cdots & n-1 & n\\
    u_r & u_{r-1} & \cdots & u_1 & v_1 & v_2 & \cdots & v_{n-r-1} & n\\
\end{pmatrix}.\]
By \cref{rmk:Sn-action}, replacing $\PT(w)$ with $\PT(\pi w)$ in \eqref{eq:shuffleproof} yields another identity. We then have
\[0= \sum_{w\in \Sh_n(x,y)} \PT(\pi w)= \sum_{w\in \Sh_n(u,v)} \PT(w)\] 
by the choice of $\pi$.
\end{proof}
\end{proposition}

\appendix
\section{The G-amplituhedron and Parke-Taylor functions}\label{sec:G}
In this section, we first give the definition of $w$-chambers in $\Ank$. Then we
describe the interpretation of some of our results in the \emph{G-amplituhedron} $\mathcal{G}_{n,k,2}$, which is a ``$Z$-independent" version of $\Ank$ introduced in \cite{LPW}. In this context, the Parke-Taylor functions $\PT(w)$ have a geometric interpretation as \emph{canonical functions} of certain \emph{positive geometries}. 

Recall the definition of twistor coordinates $\llrr{Yab}$ from \cref{def:tw_coords}.
\begin{definition}[{\cite[Definition 10.7]{LPW}}]\label{def:ampchamber}
	Fix $Z \in \Mat_{n, k+2}^{>0}$. Let $w \in D_{k+1,n}$, $I_a=\cdes(w^{(a)})$, and define
	\begin{equation*}
	(\cham{w})^{\circ}:=\bigl\{ Y \in \Gr_{k,k+2}\colon	\sgn \llrr{Y ab}= (-1)^{|I_a \cap [a, b-1]|-1} \mbox{ for all } a<b  \bigr\} 
	\end{equation*}
	The closure $\cham{w}:= \overline{(\cham{w})^{\circ}}$ in $\Gr_{k,k+2}$ is a \emph{$w$-chamber} of $\Ank$.
\end{definition}

Depending on the choice of $Z$, $\cham{w}$ may be empty \cite[Section 11.3]{PSW}, though there always exists a $Z$ such that $\cham{w}$ is nonempty \cite[Theorem 11.5]{PSW}. To avoid this irregularity, one may pass to the \emph{G-amplituhedron} and define $w$-chambers there.

	Given $v\in \R^n$, let
	$\mbox{var}(v)$ be the number of times the entries of $v$ changes sign 
	when we read the entries from left to right and ignore any
	zeros. For example, if $v:=(4,-1,0,-2)$ then $\mbox{var}(v)=1$.
\begin{definition}[{\cite[Definition 11.9]{LPW}}]
	Fix $k<n$ and let 
	\begin{align*}
		\mathcal{G}_{n,k,2}^{\circ}:= \{z\in \Gr_{2,n} \ &\colon \  
		P_{i,i+1}(z)>0 \text{ for }1 \leq i \leq n-1,
		(-1)^kP_{1n}(z) >0,\\
		&\text{ and } 
		\var((P_{12}(z), 
		P_{13}(z),  \dots 
		P_{1n}(z))=k\}.
	\end{align*}
	The closure $\mathcal{G}_{n,k,2}:=\overline{\mathcal{G}_{n,k,2}^{\circ}}$ in $\Gr_{2,n}$ is
	the \emph{G-amplituhedron}. The \emph{total G-amplituhedron} is the union
 \[\mathcal{G}_n := \bigcup_{k=0}^{n-2}\mathcal{G}_{n,k,2} \subset \Gr_{2,n}.\]
\end{definition}
One should think of $\mathcal{G}_n$ as an analogue of the hypercube: just as the hypercube is the union of all (projected) hypersimplices, the total G-amplituhedron is the union of all G-amplituhedra.

Motivated by the decomposition of $\Ank$ into $w$-chambers, we analogously define $w$-chambers for $\mathcal{G}_{n,k,2}$ which are (closures of) certain uniform oriented matroid strata in $\mbox{Gr}_{2,n}$.
\begin{definition}[{\cite[Definition 11.13]{LPW}}]\label{def:ampchamber2}
	Let $w \in D_{k+1,n}$, $I_a=\cdes(w^{(a)})$, and define
	\begin{equation*}
	(\gcham{w})^{\circ}:=\bigl\{ z \in \Gr_{2,n}\colon	\sgn P_{ ab}(z)= (-1)^{|I_a \cap [a, b-1]|-1} \mbox{ for all } a<b  \bigr\} 
	\end{equation*}
	The closure $\gcham{w}:= \overline{(\gcham{w})^{\circ}}$ in $\Gr_{2,n}$ is a \emph{$w$-chamber} of the G-amplituhedron $\mathcal{G}_{n,k,2}$.
\end{definition}

\begin{remark}
 In analogy with the amplituhedron $\mathcal{A}_{n,k,2}(Z)$, the G-amplituhedron $\mathcal{G}_{n,k,2}$ is the union of the $w$-chambers $\gcham{w}$, with $w$ ranging over  $D_{k+1,n}$ \cite[Theorem 11.21]{PSW}. The amplituhedron $\mathcal{A}_{n,k,2}(Z)$ can be seen as a $2k$-dimensional linear slice of $\mathcal{G}_{n,k,2}$ and the $w$-chamber $\cham{w}$ in $\mathcal{A}_{n,k,2}(Z)$ as a $2k$-dimensional linear slice of the $w$-chamber $\gcham{w}$ (see \cite[Proposition 11.11, Remark 11.17]{PSW}).   
\end{remark}

\begin{proposition}[{\cite[Proposition 11.15]{PSW}}] \label{proposition:GchamberNonempty}
	Let $w \in D_{k+1, n}$. Then 
	$(\gcham{w})^{\circ}$ 
	is nonempty and is contractible. 
\end{proposition}
From the proof of \cite[Proposition 11.15]{PSW}, we can find points in $(\gcham{w})^{\circ}$ explicitly. Consider $n$ vectors $v_1, v_2, \dots, v_n$ in $\mathbb{R}^2$ so that the matrix
\begin{equation}\label{eq:app_matrix}
  \begin{pmatrix}
		v_1 & v_{w_1^{(1)}}& v_{w_2^{(1)}} &\dots& v_{w_{n-1}^{(1)}}
	\end{pmatrix} \in \mbox{Mat}^{>0}_{2,n},  
\end{equation}
 equivalently, the vectors $v_1, v_{w_1^{(1)}}, v_{w_2^{(1)}}, \dots, v_{w_{n-1}^{(1)}}$ drawn in a plane are ordered counterclockwise. (Recall that $w^{(1)}$ denotes the rotation of $w$ ending in $1$, and hence $w^{(1)}_n=1$.)
	Now, set $z_1:=v_1$ and $z_b := (-1)^{|I_1 \cap [1, b-1]|-1}v_b$, for $b>2$. Then
 \begin{equation}\label{eq:ptInChamber}
		z=\begin{pmatrix}
		z_1 & z_2& z_3 &\dots& z_n
	\end{pmatrix}\end{equation}
	represents a point in $(\gcham{w})^{\circ}$. Moreover, all points in $(\gcham{w})^{\circ}$ arise this way.

\begin{example} \label{ex:realizable}
	Let $w=2564137\in D_{4,7}$, so $w^{(1)}=3725641$. 
	We have $I_1 = \{1,2,4,6\}$ and
	\begin{align*}
		(v_1, v_{w_1^{(1)}}, v_{w_2^{(1)}}, v_{w_3^{(1)}}, v_{w_4^{(1)}}, v_{w_5^{(1)}}, v_{w_6^{(1)}})
		&= (v_1, v_3, v_7, v_2, v_5, v_6, v_4)  
		= \begin{pmatrix}
			1 & 1 & 1 & 1 & 1 & 1 & 1 \\
			1 & 2 & 3 & 4 & 5 & 6 & 7
		\end{pmatrix}.
	\end{align*}
	We then get \vspace{-1em}
	\begin{equation*}
		z = \begin{pmatrix}
			1 & 1 & -1 & -1 & 1 & 1 & -1 \\
			1 & 4 & -2 & -7 & 5 & 6 & -3
		\end{pmatrix} \in (\gcham{w})^{\circ}.
	\end{equation*}
\end{example}
\begin{remark}
 The positive Grassmannian $\Gr^{\geq 0}_{2,n}$ is a \emph{positive geometry} in the sense of \cite{ABL,Lam:2022yly}. Furthermore, the Parke-Taylor function $\PT(\mathbf{I}_n)$ is the \emph{canonical function} of $\Gr^{\geq 0}_{2,n}$; that is 
multiplying $\PT(\mathbf{I}_n)$ by the standard top form of $\Gr_{2,n}$ (cf. \cite[Appendix C.2]{ABL}) gives the \emph{canonical form} of $\Gr^{\geq 0}_{2,n}$. The poles $P_{i,i+1}=0$ of $\PT(\mathbf{I}_n)$ correspond to the facets of $\Gr^{\geq 0}_{2,n}$. 

More generally, $\Gr_{k,n}^{\ge 0}$ is a positive geometry \cite{ABL,Lam:2022yly} and the amplituhedron $\A_{n, k, m}(Z)$ is conjectured to be a positive geometry. This was recently proved for $k=2, m=2$ \cite{Ranestad:2024svp}.
\end{remark}
\begin{remark}
 Each $w$-chamber $\gcham{w}$ of the G-amplituhedron is isomorphic to $\Gr^{\geq 0}_{2,n}$ by a map induced by permuting and rescaling columns \cite[Proposition 11.15]{PSW}.
Hence $\gcham{w}$ is also a positive geometry and the weight function $\Omega(\simp{w})=\PT (w)$ of the $w$-simplex $\simp{w}$ is the canonical function of $\gcham{w}$. Indeed, the poles $P_{w_i, w_i+1}=0$ of $\PT (w)$ correspond to the facets of $\gcham{w}$. From the definition of $z_a$ and \eqref{eq:app_matrix}, the facets of $\gcham{w}$ are in the locus where $v_{w^{(1)}_i},v_{w^{(1)}_{i+1}}$ are parallel and are therefore cut out by the equation $P_{w_i, w_i+1}=0$ in the Pl\"uckers of the matrix $z$ of \eqref{eq:ptInChamber}.   
\end{remark}

Given a $(k,n)$-bicolored subdivision $\sigma$, one can define \emph{tiles} in the G-amplituhedron by considering the union of the $w$-chambers $\gcham{w}$ over $(w) \in \Ext(C_\sigma)$. One can use these tiles to tile the G-amplituhedron $\mathcal{G}_{n,k,2}$ and such tilings of $\mathcal{G}_{n,k,2}$ are in bijection with tilings of $\Delta_{k+1,n}$ (and with all-$Z$ tilings of $\Ank$). 

The total G-amplituhedron $\mathcal{G}_{n}$ is given by the union of the $w$-chambers $\cham{w}$ over $w \in D_n$ \cite[Theorem 11.21]{PSW}.
Then one can consider tricolored subdivision $\tdiv$ and define the region in $\mathcal{G}_{n}$ which is the union of the $w$-chambers $\cham{w}$ with $w \in \Ext(C_\tdiv)$. This is the analogue of the Parke-Taylor polytope $\ptile{\tdiv}$. As the weight function $\Omega(\simp{w})$ of $w$-simplices gives the canonical function of $w$-chambers $\Delta^\mathcal{G}_w$, it would be interesting to interpret the weight function of Parke-Taylor polytopes $\Omega(\ptile{\tdiv})$ as canonical functions of some positive geometries inside the (total) G-amplituhedron.

\bibliographystyle{alpha}
	\bibliography{ClusterTilesPromotion}
\end{document}